\documentclass{article}

\usepackage{amssymb,amsthm,amsmath}

\newcommand{\ZZ}{\mathbb{Z}}

\newcommand{\NNo}{\mathbb{N}_0}

\newcommand{\FF}{\mathbb{F}}

\newtheorem{theorem}{Theorem}[section]
\newtheorem{lemma}[theorem]{Lemma}
\newtheorem{proposition}[theorem]{Proposition}
\newtheorem{corollary}[theorem]{Corollary}

\theoremstyle{definition}
\newtheorem{definition}[theorem]{Definition}
\newtheorem{example}[theorem]{Example}

\usepackage[square,numbers]{natbib}
\usepackage{hyperref}
\usepackage{breakurl}
\usepackage{url}
\hypersetup{pdfpagemode=none,colorlinks,citecolor=blue,filecolor=black,linkcolor=black,urlcolor=blue}
\newcommand{\vct}[1]{\mathbf{#1}}
\newcommand{\zint}{\,..\,}
\newcommand{\nth}{^{\text{th}}}
\DeclareMathOperator{\moddec}{mod}
\renewcommand{\mod}[1]{\,(\moddec #1)}
\usepackage{enumitem}
\usepackage{nameref}
\makeatletter
\newcommand{\clabel}[2]{\protected@write \@auxout {}{\string \newlabel {#1}{{#2}{\thepage}{#2}{#1}{}} }\hypertarget{#1}{}}
\makeatother
\newcommand{\lb}{\allowbreak}
\usepackage{etoolbox}
\newcommand{\breaklist}[2][,\lb]{\def\nextitem{\def\nextitem{#1}}\renewcommand*{\do}[1]{\nextitem{##1}}\docsvlist{#2}}
\DeclareMathOperator{\cirdec}{circ}		
\DeclareMathOperator{\CIRdec}{CIRC}	
\newcommand{\cir}[1]{\cirdec(\breaklist{#1})}
\newcommand{\CIR}[1]{\CIRdec(\breaklist{#1})}
\newcommand{\floor}[1]{\lfloor#1\rfloor}
\usepackage{empheq}

\DeclareMathOperator{\autdec}{Aut}
\newcommand{\aut}[1]{\autdec(#1)}
\usepackage{booktabs}
\usepackage{upgreek}

\makeatletter
\renewcommand*\env@matrix[1][*\c@MaxMatrixCols c]{\hskip -\arraycolsep\let\@ifnextchar\new@ifnextchar\array{#1}}
\makeatother
\usepackage[font=footnotesize,labelfont=bf,textfont=normalfont,margin=0cm]{caption}
\usepackage{adjustbox}
\providecommand{\keywords}[1]{\small\textit{Keywords}: #1}
\providecommand{\msc}[1]{\small\textit{2020 MSC}: #1}

\title{New binary self-dual codes of lengths 80, 84 and 96 from composite matrices}

\author{J. Gildea, A. Korban and A. M. Roberts\\
Department of Mathematical and Physical Sciences\\
University of Chester\\
Thornton Science Park\\
Chester CH2 4NU\\
United Kingdom
}

\date{}

\begin{document}

\maketitle


\msc{94B05, 16S34, 15B10, 15B33}

\let\thefootnote\relax\footnote{E-mail addresses: \href{mailto:j.gildea@chester.ac.uk}{j.gildea@chester.ac.uk} (J. Gildea), \href{mailto:adrian3@windowslive.com}{adrian3@windowslive.com} (A. Korban), \href{mailto:adammichaelroberts@outlook.com}{adammichaelroberts@outlook.com} (A. M. Roberts)}

\begin{abstract}
In this work, we apply the idea of composite matrices arising from group rings to derive a number of different techniques for constructing self-dual codes over finite commutative Frobenius rings. By applying these techniques over different alphabets, we construct best known singly-even binary self-dual codes of lengths 80, 84 and 96 as well as doubly-even binary self-dual codes of length 96 that were not known in the literature before. 
\end{abstract}

\section{Introduction}

Self-dual codes form a family of widely studied linear codes which have many interesting properties and are intimately connected with many mathematical structures such as designs, lattices, modular forms and sphere packings. In recent history, much work has particularly been invested in developing techniques to construct extremal and optimal binary self-dual codes. The most famous of these techniques is quite possibly the pure double circulant construction, which utilises a generator matrix of the form $G=(I_n\,|\,A)$ where $I_n$ is the $n\times n$ identity matrix and $A$ is an $n\times n$ circulant matrix. It follows that $G$ is a generator matrix of a self-dual $[2n,n]$ code if and only if $AA^T=-I_n$. This technique has since been generalised by assuming a generator matrix of the form $G=(I_n\,|\,\sigma(v))$ where $\sigma$ is an isomorphism from a group ring to the ring of matrices which was introduced in \cite{R-022}. The isomorphism $\sigma$ is such that $G$ is a generator matrix of a self-dual $[2n,n]$ code if and only if $v$ is a unitary unit in the group ring. See \cite{R-096,R-111,R-100,R-113} for recent applications of this isomorphism in constructing self-dual codes.

In this work, we assume a generator matrix of the form $G=(I_n\,|\,\Omega(v))$ where $\Omega(v)$ is a matrix that arises from group rings which we call a composite matrix. It clearly follows that $(I_n\,|\,\Omega(v))$ is a generator matrix of a self-dual code if and only if $\Omega(v)\Omega(v)^T=-I_n$. The idea of composite matrices was first introduced in \cite{R-153} as a way of generalising the structure of $\sigma(v)$. The primary motivation for employing this technique is obtaining codes whose structures are atypical compared with those of codes constructed by more classical techniques. The main problem we face when attempting to construct codes with such a generator matrix is choosing parameters in such a way that allows for structural complexity of $\Omega(v)$ while also allowing for a reasonable set of necessary and sufficient conditions for the satisfaction of $\Omega(v)\Omega(v)^T=-I_n$.

Using generator matrices of the form $(I_n\,|\,\Omega(v))$ for a number of different composite matrices $\Omega(v)$, we find many self-dual codes with weight enumerator parameters of previously unknown values (relative to referenced sources). In total, 361 new codes are found, including 28 singly-even binary self-dual $[80,40,14]$ codes, 107 binary self-dual $[84,42,14]$ codes, 105 singly-even binary self-dual $[96,48,16]$ codes and 121 doubly-even binary self-dual $[96,48,16]$ codes.

The rest of the work is organised as follows. In Section 2, we give preliminary definitions on self-dual codes, Gray maps, circulant matrices and the alphabets we use. We also prove a few results concerning a simple matrix transformation, which we use in two of the composite matrix definitions. In Section 3, we define the composite matrices which we utilise in our constructions and we also prove the necessary and sufficient conditions needed by each construction to produce a self-dual code. In Section 4, we apply the constructions to obtain the new self-dual codes of length 80, 84 and 96 whose weight enumerator parameter values and automorphism group orders we detail. We also tabulate the results in this section. We finish with concluding remarks and discussion of possible expansion on this work.

\section{Preliminaries}

\subsection{Self-Dual Codes}

Let $R$ be a commutative Frobenius ring. Throughout this work, we always assume $R$ has unity. A code $\mathcal{C}$ of length $n$ over $R$ is a subset of $R^n$ whose elements are called codewords. If $\mathcal{C}$ is a submodule of $R^n$, then we say that $\mathcal{C}$ is linear. Let $\mathbf{x},\mathbf{y}\in R^n$ where $\mathbf{x}=(x_1,x_2,\dots,x_n)$ and $\mathbf{y}=(y_1,y_2,\dots,y_n)$. The (Euclidean) dual $\mathcal{C}^{\bot}$ of $\mathcal{C}$ is given by
	\begin{equation*}
	\mathcal{C}^{\bot}=\{\mathbf{x}\in R^n: \langle\mathbf{x},\mathbf{y}\rangle=0,\forall\mathbf{y}\in\mathcal{C}\},
	\end{equation*}	
where $\langle,\rangle$ denotes the Euclidean inner product defined by
	\begin{equation*}
	\langle\mathbf{x},\mathbf{y}\rangle=\sum_{i=1}^nx_iy_i.
	\end{equation*}

We say that $\mathcal{C}$ is self-orthogonal if $\mathcal{C}\subseteq \mathcal{C}^\perp$ and self-dual if $\mathcal{C}=\mathcal{C}^{\bot}$.

A binary self-dual code $\mathcal{C}$ is said to be \textit{doubly-even} (Type II), if all codewords $\vct{c}\in\mathcal{C}$ have weight $w(\vct{c})\equiv 0\mod{4}$, otherwise $\mathcal{C}$ is said to be \textit{singly-even} (Type I).

An upper bound on the minimum (Hamming) distance of a doubly-even binary self-dual code was given in \cite{R-116} and likewise for a singly-even binary self-dual code in \cite{R-115}. Let $d_{\text{I}}(n)$ and $d_{\text{II}}(n)$ be the minimum distance of a singly-even and doubly-even binary self-dual code of length $n$, respectively. Then
	\begin{equation*}
	d_{\text{II}}(n)\leq 4\floor{n/24}+4
	\end{equation*}
and
	\begin{equation*}
	d_{\text{I}}(n)\leq
	\begin{cases}
	4\floor{n/24}+2,& n\equiv 0\mod{24},\\
	4\floor{n/24}+4,& n\not\equiv 22\mod{24},\\
	4\floor{n/24}+6,& n\equiv 22\mod{24}.
	\end{cases}
	\end{equation*}

A self-dual code whose minimum distance meets its corresponding bound is called \textit{extremal}. A self-dual code with the highest possible minimum distance for its length is said to be \textit{optimal}. Extremal codes are necessarily optimal but optimal codes are not necessarily extremal. A \textit{best known} self-dual code is a self-dual code with the highest known minimum distance for its length.

\subsection{Alphabets}

In this paper, we consider the alphabets $\FF_2$, $\FF_2+u\FF_2$ and $\FF_4$.

Define
	\begin{equation*}
	\FF_2+u\FF_2=\{a+bu:a,b\in\FF_2,u^2=0\}.
	\end{equation*}

Then $\FF_2+u\FF_2$ is a commutative ring of order 4 and characteristic 2 such that $\FF_2+u\FF_2\cong\FF_2[u]/\langle u^2\rangle$. 

We define $\FF_4\cong\FF_2[\omega]/\langle \omega^2+\omega+1\rangle$ so that
	\begin{equation*}
	\FF_4=\{a{\omega}+b(1+\omega): a,b\in\FF_2,\omega^2+\omega+1=0\}.
	\end{equation*}

We recall the following Gray maps from \cite{R-117,R-118}:
	\begin{align*}
	\varphi_{\FF_2+u\FF_2}&:(\FF_2+u\FF_2)^n\to\FF_2^{2n}\\
        &\quad a+bu\mapsto(b,a+b),\,a,b\in\FF_2^{n},\\[6pt]
	\psi_{\FF_4}&:\FF_4^n\to\FF_2^{2n}\\
		&\quad a\omega+b(1+\omega)\mapsto(a,b),\,a,b\in\FF_2^n.
	\end{align*}

Note that these Gray maps preserve orthogonality in their respective alphabets. The Lee weight of a codeword is defined to be the Hamming weight of its binary image under any of the aforementioned Gray maps. A self-dual code in $R^n$ where $R$ is equipped with a Gray map to the binary Hamming space is said to be of Type II if the Lee weights of all codewords are multiples of 4, otherwise it is said to be of Type I. 
	\begin{proposition}\textup{(\cite{R-117})}\label{proposition-F2u}
		Let $\mathcal{C}$ be a code over $\FF_2+u\FF_2$. If $\mathcal{C}$ is self-orthogonal, then $\varphi_{\FF_2+u\FF_2}(\mathcal{C})$ is self-orthogonal. The code $\mathcal{C}$ is a Type I (resp. Type II) code over $\FF_2+u\FF_2$ if and only if $\varphi_{\FF_2+u\FF_2}(\mathcal{C})$ is a Type I (resp. Type II) code over $\FF_2$. The minimum Lee weight of $\mathcal{C}$ is equal to the minimum Hamming weight of $\varphi_{\FF_2+u\FF_2}(\mathcal{C})$.
	\end{proposition}
	\begin{proposition}\textup{(\cite{R-118})}\label{proposition-F4}
		Let $\mathcal{C}$ be a code over $\FF_4$. If $\mathcal{C}$ is self-orthogonal, then $\psi_{\FF_4}(\mathcal{C})$ is self-orthogonal. The code $\mathcal{C}$ is a Type I (resp. Type II) code over $\FF_4$ if and only if $\psi_{\FF_4}(\mathcal{C})$ is a Type I (resp. Type II) code over $\FF_2$. The minimum Lee weight of $\mathcal{C}$ is equal to the minimum Hamming weight of $\psi_{\FF_4}(\mathcal{C})$.
	\end{proposition}

The next two corollaries follow directly from Propositions \ref{proposition-F2u} and \ref{proposition-F4}, respectively.
	\begin{corollary}
		Let $\mathcal{C}$ be a self-dual code over $\FF_2+u\FF_2$ of length $n$ and minimum Lee distance $d$. Then $\varphi_{\FF_2+u\FF_2}(\mathcal{C})$ is a binary self-dual $[2n,n,d]$ code. Moreover, the Lee weight enumerator of $\mathcal{C}$ is equal to the Hamming weight enumerator of $\varphi_{\FF_2+u\FF_2}(\mathcal{C})$. If $\mathcal{C}$ is a Type I (resp. Type II) code, then $\varphi_{\FF_2+u\FF_2}(\mathcal{C})$ is a Type I (resp. Type II) code.
	\end{corollary}
	\begin{corollary}
		Let $\mathcal{C}$ be a self-dual code over $\FF_4$ of length $n$ and minimum Lee distance $d$. Then $\psi_{\FF_4}(\mathcal{C})$ is a binary self-dual $[2n,n,d]$ code. Moreover, the Lee weight enumerator of $\mathcal{C}$ is equal to the Hamming weight enumerator of $\psi_{\FF_4}(\mathcal{C})$. If $\mathcal{C}$ is a Type I (resp. Type II) code, then $\psi_{\FF_4}(\mathcal{C})$ is a Type I (resp. Type II) code.
	\end{corollary}

\subsection{Special Matrices}

We now recall the definitions and properties of some special matrices which we use in our work. We begin by defining a matrix transformation whose properties we utilise in some of the composite constructions we consider. The properties are easy to prove but we do so for completeness.
	\begin{proposition}\label{proposition-star}
		Let $A$ be $n\times n$ matrix over a commutative ring $R$. Let $\star:R^{n\times n}\to R^{n\times n}$ be the transformation such that $A^{\star}$ is defined to be the matrix obtained after circularly shifting the columns of $A$ to the right by one position. If
			\begin{align*}
				P=\begin{pmatrix}
					\vct{0} & I_{n-1}\\
					1 & \vct{0}
				\end{pmatrix},
			\end{align*}
		then $A^{\star}=AP$.
	\end{proposition}
	\begin{proof}
		Assume $A$ is an $n\times n$ matrix over a commutative ring $R$ where $n\geq 2$. Suppose we decompose $A$ into blocks such that
			\begin{align*}
				A=\begin{pmatrix}
					\vct{x} & z\\
					X & \vct{y}^T
				\end{pmatrix}
			\end{align*}
		where $\vct{x},\vct{y}\in R^{1\times(n-1)}$, $z\in R$ and $X\in R^{(n-1)\times(n-1)}$. Then by block-wise multiplication we obtain
			\begin{align*}
				AP=\begin{pmatrix}
					\vct{x} & z\\
					X & \vct{y}^T
				\end{pmatrix}
				\begin{pmatrix}
					\vct{0} & I_{n-1}\\
					1 & \vct{0}
				\end{pmatrix}=
				\begin{pmatrix}
					z & \vct{x}\\
					\vct{y}^T & X
				\end{pmatrix}
			\end{align*}	
		and so $AP$ corresponds to $A$ after circularly shifting its columns to the right by one position. Thus, $A^{\star}=AP$.
	\end{proof}

The matrix $P$ as defined in Proposition \ref{proposition-star} is a permutation matrix and is therefore orthogonal, i.e. $PP^T=I_n$. To see this, we have
	\begin{align*}
		P^T=\begin{pmatrix}
			\vct{0} & 1\\
			I_{n-1} & \vct{0}
		\end{pmatrix},
	\end{align*}
which corresponds to $P$ after circularly shifting its columns to the right by $n-2$ places and so by Proposition \ref{proposition-star} we have $P{}^T=PP^{n-2}=P^{n-1}$. Clearly, if we circularly shift the columns of $P{}^T=P^{n-1}$ to the right by one place we obtain $I_n$ so that $P{}^TP=P^{n-1}P=P^n=I_n$.

It also follows that $(P^k){}^T=P^{-k\mod{n}}$ for $k\in\NNo$. We can easily prove this by induction on $k\in\NNo$. The cases $k=0$ and $k=1$ are trivial. Assume $(P^k){}^T=P^{-k\mod{n}}$. Then we have $(P^{k+1}){}^T=(P^kP){}^T=P{}^T(P^k){}^T=P^{n-1}P^{-k\mod{n}}=P^{n-k-1\mod{n}}=P^{-(k+1)\mod{n}}$ which concludes our induction step.

We also have the following properties which are easy to prove.
	\begin{lemma}\label{lemma-star}
		Let $A$ and $B$ be $n\times n$ matrices over a commutative ring $R$ where $n\geq 2$ and let $\star$ be the transformation defined in Proposition \ref{proposition-star}.
		\begin{enumerate}
		\item[$(i)$]\clabel{lemma-star-1}{$(i)$} $(A+B)^{\star}=A^{\star}+B^{\star}$.
		\item[$(ii)$]\clabel{lemma-star-2}{$(ii)$} $AB{}^T=A^{\star}B^{\star T}$.
		\end{enumerate}
	\end{lemma}	
	\begin{proof}
		$(i)$. By Proposition \ref{proposition-star}, we have $(A+B)^{\star}=(A+B)P=AP+BP=A^{\star}+B^{\star}$.
		
		$(ii)$. By Proposition \ref{proposition-star} and the fact that $P$ is orthogonal, we have $A^{\star}B^{\star T}=AP(BP){}^T=APP{}^TB{}^T=A(I_n)B{}^T=AB{}^T$.
	\end{proof}

Let $\vct{a}=(a_0,a_1,\ldots,a_{n-1})\in R^n$ where $R$ is a commutative ring and let
	\begin{equation*}
	A=\begin{pmatrix}
	a_0 & a_1 & a_2 & \cdots & a_{n-1}\\
    a_{n-1} & a_0 & a_1 & \cdots & a_{n-2}\\
    a_{n-2} & a_{n-1} & a_0 & \cdots & a_{n-3}\\
	\vdots & \vdots & \vdots & \ddots & \vdots\\
	a_1 & a_2 & a_3 & \cdots & a_0
	\end{pmatrix}.
	\end{equation*}
	
Then $A$ is an $n\times n$ matrix called the \textit{circulant} matrix generated $\vct{a}$, denoted by $A=\cir{\vct{a}}$.

If $A=\cir{a_0,a_1,\ldots,a_{n-1}}$, then we see that $A=a_0I_n+a_1I_n^{\star}+a_2(I_n^{\star})^{\star}+\ldots$ and so on. Using Proposition \ref{proposition-star} and the properties of the matrix $P$, it follows that $A=\sum_{i=0}^{n-1}a_iP^i$. Clearly, the sum of any two circulant matrices is also a circulant matrix. If $B=\cir{\vct{b}}$ where $\vct{b}=(b_0,b_1,\ldots,b_{n-1})\in R^n$, then $AB=\sum_{i=0}^{n-1}\sum_{j=0}^{n-1}a_ib_jP^{i+j}$. Since $P^n=I_n$ there exist $c_k\in R$ such that $AB=\sum_{k=0}^{n-1}c_kP^k$ so that $AB$ is also circulant. In fact, it is true that
	\begin{equation*}
	c_{k}=\sum_{[i+j]_n=k}a_ib_j=\vct{x}_1\vct{y}_{k+1}
	\end{equation*}
for $k\in[0\zint n-1]$, where $\vct{x}_i$ and $\vct{y}_i$ respectively denote the $i\nth$ row and column of $A$ and $B$ and $[i+j]_n$ denotes the smallest non-negative integer such that $[i+j]_n\equiv i+j\mod{n}$. From this, we can see that circulant matrices commute multiplicatively. We also see that $A^T$ is circulant such that $A^T=\sum_{i=0}^{n-1}a_i(P^i)^T=\sum_{i=0}^{n-1}a_iP^{n-i}$. 
	\begin{lemma}\label{lemma-starcirc}
		Let $A$ be an $n\times n$ matrix over a commutative ring $R$ where $n\geq 2$ and let $\star$ be the transformation defined in Proposition \ref{proposition-star}. Let $B$ be an $n\times n$ circulant matrix over $R$.
		\begin{enumerate}
		\item[$(i)$]\clabel{lemma-starcirc-1}{$(i)$} $BP=PB$.
		\item[$(ii)$]\clabel{lemma-starcirc-2}{$(ii)$} $(AB{}^T)^{\star}=A^{\star}B{}^T$.
		\item[$(iii)$]\clabel{lemma-starcirc-2}{$(iii)$} $(AB^{\star T})^{\star}=AB{}^T$.
		\end{enumerate}
	\end{lemma}	
	\begin{proof}
		$(i)$. Let $B=\cir{b_0,b_1,\ldots,b_{n-1}}$. Then $B$ can be expressed as $B=\sum_{i=0}^{n-1}b_iP^i$ and so it is obvious that $BP=PB$.
		
		$(ii)$. Since $B$ is circulant, then $B^T$ is circulant and so by \ref{lemma-starcirc-1}, we have $(AB^T)^{\star}=(AB^T)P=A(B^TP)=(AP)B^T=A^{\star}B^T$.
		
		$(iii)$. Since $B$ is circulant, then $B^T$ is circulant and so by \ref{lemma-starcirc-1} and the fact that $P$ is orthogonal, we have $(AB^{\star T})^{\star}=(A(BP)^T)P=AP^TB^TP=AP^TPB^T=A(I_n)B^T=AB^T$.
	\end{proof}

Let $J_n$ be an $n\times n$ matrix over $R$ whose $(i,j)\nth$ entry is $1$ if $i+j=n+1$ and 0 if otherwise. Then $J_n$ is called the $n\times n$ exchange matrix and corresponds to the row-reversed (or column-reversed) version of $I_n$. Note that $[i+j]_n$ corresponds to the $(i+1,j+1)\nth$ entry of the matrix $J_nV$ where $V=\cir{n-1,0,1,\ldots,n-2}$ for $i,j\in[0\zint n-1]$.

Let $A_0,A_1,\ldots,A_{k-1}$ be $m\times n$ matrices over $R$ and let
	\begin{equation*}
	X=\begin{pmatrix}
	A_0 & A_1 & A_2 & \cdots & A_{k-1}\\
    A_{k-1} & A_0 & A_1 & \cdots & A_{k-2}\\
    A_{k-2} & A_{k-1} & A_0 & \cdots & A_{k-3}\\
	\vdots & \vdots & \vdots & \ddots & \vdots\\
	A_1 & A_2 & A_3 & \cdots & A_0
	\end{pmatrix}.
	\end{equation*}
	
Then $X$ is an $km\times kn$ matrix called the \textit{block circulant} matrix generated $A_0,A_1,\ldots,A_{k-1}$, denoted by $X=\CIR{A_0,A_1,\ldots,A_{k-1}}$.
\subsection{Group Rings and Composite Matrices}

In this section, we recall the basic definition of a finite group ring and proceed to define the concept of a composite matrix.

Let $G$ be a finite group order $n$ and let $R$ be a finite commutative Frobenius ring. Let $RG=\{\sum_{i=1}^n\alpha_{g_i}g_i:\alpha_{g_i}\in R,g_i\in G\}$ and define addition in $RG$ by
	\begin{align*}
		\sum_{i=1}^n\alpha_{g_i}g_i+\sum_{i=1}^n\beta_{g_i}g_i=\sum_{i=1}^n(\alpha_{g_i}+\beta_{g_i})g_i
	\end{align*}
and define multiplication in $RG$ by
	\begin{align*}
		\sum_{i=1}^n\alpha_{g_i}g_i\cdot\sum_{j=1}^n\beta_{g_j}g_j=\sum_{k=1}^n\left(\sum\nolimits_{i,j:g_ig_j=g_k}\alpha_{g_i}\beta_{g_j}\right)g_k.
	\end{align*}

Then $RG$ is called the group ring of $G$ over $R$ and is a ring with respect to the aforementioned definitions of addition and multiplication.

Let $(g_1,g_2,\ldots,g_n)$ be a fixed listing of the elements of $G$ with $g_1=1$ and let $v=\sum_{i=1}^n\alpha_{g_i}g_i\in RG$. Define $\sigma(v)$ to be the $n\times n$ matrix whose $(i,j)\nth$ entry is $\alpha_{g_k}$ where $g_k=g_i^{-1}g_j$ for $i,j\in[1\zint n]$, i.e.
	\begin{align*}
		\sigma(v)=\begin{pmatrix}
			\alpha_{g_1^{-1}g_1} & \alpha_{g_1^{-1}g_2} & \cdots & \alpha_{g_1^{-1}g_n}\\
			\alpha_{g_2^{-1}g_1} & \alpha_{g_2^{-1}g_2} & \cdots & \alpha_{g_2^{-1}g_n}\\
			\vdots & \vdots & \ddots & \vdots\\
			\alpha_{g_n^{-1}g_1} & \alpha_{g_n^{-1}g_2} & \cdots & \alpha_{g_n^{-1}g_n}
		\end{pmatrix}.
	\end{align*}
	
The matrix $\sigma(v)$ was first given in \cite{R-022} wherein it was proved that $\sigma$ is an isomorphism from the ring $RG$ to $R^{n\times n}$. 

Suppose now that $n>1$ is composite and let $r$ be a fixed integer such that $r\mid n:1<r<n$ and let $m=n/r$. Let $\{H_1,H_2,\ldots,H_{\eta}\}$ be a collection of $\eta$ groups of order $r$. Let $H_t$ be a representative of one of these groups for $t\in[1\zint\eta]$ and let $(h_{t:1},h_{t:2},\ldots,h_{t:r})$ be a fixed listing of the elements of $H_t$ with $h_{t:1}=1$. Let $H'$ be an $m\times m$ matrix whose $(y,z)\nth$ entry is $h_{y,z}'\in[1\zint\eta]$ for $y,z\in[1\zint m]$ and let $P'$ be an $m\times m$ matrix whose $(y,z)\nth$ entry is $p_{y,z}'\in\FF_2$ for $y,z\in[1\zint m]$. Define the mapping $\varrho(y,z,i,j)=g_{r(y-1)+i}^{-1}g_{r(z-1)+j}$ for $y,z\in[1\zint m]$ and $i,j\in[1\zint r]$.

Define $Z_{y,z}$ to be the $r\times r$ matrix whose $(i,j)\nth$ entry is given by
	\begin{align*}
		z_{y,z|i,j}=\alpha_{\varrho(y,z,i,j)}
	\end{align*}	
and define $Z'_{t:y,z}$ to be the $r\times r$ matrix whose $(i,j)\nth$ entry is given by
	\begin{align*}
		z_{t:y,z|i,j}'=\alpha_{\varrho(y,z,1,\mathcal{M}_{H_t}(i,j))},
	\end{align*}	
where $\mathcal{M}_{H_t}(i,j)$ is the $(i,j)\nth$ entry of the matrix of integers $\ell\in[1\zint r]$ such that $h_{t:\ell}=h_{t:i}^{-1}h_{t:j}$.

Define $\Omega(v)$ to be the block matrix whose $(y,z)\nth$ block entry is given by
	\begin{align*}
		\omega_{y,z}=\begin{cases}
		Z_{y,z},&p_{y,z}'=0,\\
		Z_{h_{y,z}':y,z}',&p_{y,z}'=1.
		\end{cases}
	\end{align*}

Then $\Omega(v)$ is an $n\times n$ matrix composed of $m^2$ blocks of size $r\times r$ which we call the \textit{composite} $(G,H_1,H_2,\ldots,H_{\eta})$-\textit{matrix of} $v\in RG$ with respect to $H'$ and $P'$. If $P'=\vct{0}$ (i.e. the $m\times m$ zero matrix), the matrix $\Omega(v)$ reduces to $\sigma(v)$. 

The concept of composite matrices defined in this way was first introduced in \cite{R-153} as a way of generalising the structure of $\sigma(v)$. See \cite{R-151,R-101,R-158} for recent applications of composite matrices in constructing binary self-dual codes.
	\begin{example}
	Let $G\cong D_4\cong\langle a,b\mid a^4=b^2=1,bab=a^{-1}\rangle$ with the fixed listing $G=(g_{4j+i+1})=a^ib^j$ for $i\in[0\zint 3]$ and $j\in[0\zint 1]$. Then $n=8$ and suppose we let $r=4\mid n$ so that $m=n/r=2$. Let $\{H_1,H_2\}$ be a collection of groups of order $r=4$. Let $H_1\cong C_2\times C_2\cong\langle c,d\mid c^2=d^2=1,cd=dc\rangle$ with the fixed listing $H_1=(h_{1:2j+i+1})=c^id^j$ for $i\in[0\zint 1]$ and $j\in[0\zint 1]$. Let $H_2\cong C_{2\cdot 2}\cong\langle e\mid e^{2\cdot 2}=1\rangle$ with the fixed listing $H_2=(h_{2:2j+i+1})=e^{2i+j}$ for $i\in[0\zint 1]$ and $j\in[0\zint 1]$. Let
		\begin{align*}
		    H'=\begin{pmatrix}
		    1 & 2\\
		    2 & 1
		    \end{pmatrix}
		\end{align*}	
	and let $P'=\vct{1}$ (i.e. the $2\times 2$ matrix of ones).	Let $v=\sum_{i=1}^{8}\alpha_{g_i}g_i\in RG$ and let $\Omega(v)$ be the composite $(G,H_1,H_2)$-matrix of $v\in RG$ with respect to $H'$ and $P'$. We have
		\begin{align*}
		    \Omega(v)=
		    \begin{pmatrix}
		    \omega_{1,1} & \omega_{1,2}\\
		    \omega_{2,1} & \omega_{2,2}
		    \end{pmatrix}
		    =
		    \begin{pmatrix}
		    Z_{h_{1,1}':1,1}' & Z_{h_{1,2}':1,2}'\\
		    Z_{h_{2,1}':2,1}' & Z_{h_{2,2}':2,2}'
		    \end{pmatrix}
		    =
		    \begin{pmatrix}
		    Z_{1:1,1}' & Z_{2:1,2}'\\
		    Z_{2:2,1}' & Z_{1:2,2}'
		    \end{pmatrix}
		\end{align*}
	and we also find that
		\begin{align*}
		    \mathcal{M}_{H_1}=\begin{pmatrix}
			1 & 2 & 3 & 4\\
			2 & 1 & 4 & 3\\
			3 & 4 & 1 & 2\\
			4 & 3 & 2 & 1
		    \end{pmatrix}\quad\text{and}\quad
		    \mathcal{M}_{H_2}=\begin{pmatrix}
			1 & 2 & 3 & 4\\
			2 & 1 & 4 & 3\\
			4 & 3 & 1 & 2\\
			3 & 4 & 2 & 1
		    \end{pmatrix}.
		\end{align*}
	
	By definition, the $(i,j)\nth$ entry of $Z_{1:1,1}'$ is given by $\alpha_{\varrho(1,1,1,\mathcal{M}_{H_1}(i,j))}$ where $\varrho(1,1,1,\mathcal{M}_{H_1}(i,j))=g_{1}^{-1}g_{\mathcal{M}_{H_1}(i,j)}=g_{\mathcal{M}_{H_1}(i,j)}$ so that
		\begin{align*}
		    Z_{1:1,1}'=
		    \begin{pmatrix}
				\alpha_{g_{1}} & \alpha_{g_{2}} & \alpha_{g_{3}} & \alpha_{g_{4}}\\
				\alpha_{g_{2}} & \alpha_{g_{1}} & \alpha_{g_{4}} & \alpha_{g_{3}}\\
				\alpha_{g_{3}} & \alpha_{g_{4}} & \alpha_{g_{1}} & \alpha_{g_{2}}\\
				\alpha_{g_{4}} & \alpha_{g_{3}} & \alpha_{g_{2}} & \alpha_{g_{1}}
		    \end{pmatrix}
		\end{align*}
	and similarly we find that
		\begin{align*}
		    Z_{2:1,2}'&=
		    \begin{pmatrix}
				\alpha_{g_{5}} & \alpha_{g_{6}} & \alpha_{g_{7}} & \alpha_{g_{8}}\\
				\alpha_{g_{6}} & \alpha_{g_{5}} & \alpha_{g_{8}} & \alpha_{g_{7}}\\
				\alpha_{g_{8}} & \alpha_{g_{7}} & \alpha_{g_{5}} & \alpha_{g_{6}}\\
				\alpha_{g_{7}} & \alpha_{g_{8}} & \alpha_{g_{6}} & \alpha_{g_{5}}
		    \end{pmatrix},\\[4pt]
		    Z_{2:2,1}'&=
		    \begin{pmatrix}
				\alpha_{g_{5}} & \alpha_{g_{8}} & \alpha_{g_{7}} & \alpha_{g_{6}}\\
				\alpha_{g_{8}} & \alpha_{g_{5}} & \alpha_{g_{6}} & \alpha_{g_{7}}\\
				\alpha_{g_{6}} & \alpha_{g_{7}} & \alpha_{g_{5}} & \alpha_{g_{8}}\\
				\alpha_{g_{7}} & \alpha_{g_{6}} & \alpha_{g_{8}} & \alpha_{g_{5}}
		    \end{pmatrix},\\[4pt]
		    Z_{1:2,2}'&=
		    \begin{pmatrix}
				\alpha_{g_{1}} & \alpha_{g_{4}} & \alpha_{g_{3}} & \alpha_{g_{2}}\\
				\alpha_{g_{4}} & \alpha_{g_{1}} & \alpha_{g_{2}} & \alpha_{g_{3}}\\
				\alpha_{g_{3}} & \alpha_{g_{2}} & \alpha_{g_{1}} & \alpha_{g_{4}}\\
				\alpha_{g_{2}} & \alpha_{g_{3}} & \alpha_{g_{4}} & \alpha_{g_{1}}
		    \end{pmatrix}.
		\end{align*}
	
	Therefore, we obtain
		\begin{align*}
		    \Omega(v)=
		    \begin{pmatrix}
		    Z_{1:1,1}' & Z_{2:1,2}'\\
		    Z_{2:2,1}' & Z_{1:2,2}'
		    \end{pmatrix}=
		    \begin{pmatrix}
		    A_1 & A_2 & B_1 & B_2\\
		    A_2 & A_1 & B_2J_2 & B_1\\
		    C_1 & C_2 & D_1 & D_2\\
		    C_2J_2 & C_1 & D_2 & D_1
		    \end{pmatrix}
		\end{align*}
	where $A_1=\cir{\alpha_{g_{1}},\alpha_{g_{2}}}$, $A_2=\cir{\alpha_{g_{3}},\alpha_{g_{4}}}$, $B_1=\cir{\alpha_{g_{5}},\alpha_{g_{6}}}$, $B_2=\cir{\alpha_{g_{7}},\alpha_{g_{8}}}$, $C_1=\cir{\alpha_{g_{5}},\alpha_{g_{8}}}$, $C_2=\cir{\alpha_{g_{7}},\alpha_{g_{6}}}$ and $D_1=\cir{\alpha_{g_{1}},\alpha_{g_{4}}}$, $D_2=\cir{\alpha_{g_{3}},\alpha_{g_{2}}}$.
	\end{example}

\section{Composite Matrix Constructions}
In this section, we present our constructions which assume a generator matrix of the form $(I_n\,|\,\Omega(v))$ where $\Omega(v)$ is a composite matrix. For each construction, we first define the structure of the corresponding composite matrix $\Omega(v)$ and subsequently prove the conditions that hold if and only if $(I_n\,|\,\Omega(v))$ is a generator matrix of a self-dual $[2n,n]$ code over $R$. We will hereafter assume that $R$ is a finite commutative Frobenius ring of characteristic 2. For each $v=\sum_{i=1}^n\alpha_{g_i}g_i\in RG$ that we define, we denote $\vct{v}=(v_1,v_2,\ldots,v_n)=(\alpha_{g_1},\alpha_{g_2},\ldots,\alpha_{g_n})$ where $\vct{v}_i$ denotes $v_i=\alpha_{g_i}$ for $i\in[1\zint n]$. We also use the following notation
	\begin{align*}
		\vct{v}_{i:j}=\begin{cases}
			(v_i,v_{i+1},v_{i+2},\ldots,v_{j-1},v_j),& i<j,\\
			(v_i,v_{i-1},v_{i-2},\ldots,v_{j+1},v_j),& i>j,
		\end{cases}
	\end{align*}
for $i,j\in[1\zint n]$. We also let $\cir{\vct{u},\vct{v}}$ denote $\cir{u_1,u_2,\ldots,u_n,v_1,v_2,\ldots,v_n}$ for any $\vct{u},\vct{v}\in R^n$ such that $\vct{u}=(u_1,u_2,\ldots,u_n)$ and $\vct{v}=(v_1,v_2,\ldots,v_n)$.
	\begin{definition}\label{definition-Omega_20_1}
	Let $G\cong D_{10}\cong\langle a,b\mid a^{10}=b^2=1,bab=a^{-1}\rangle$ with the fixed listing $G=(g_{10j+i+1})=a^ib^j$ for $i\in[0\zint 9]$ and $j\in[0\zint 1]$. Let $H\cong D_{5}\cong\langle c,d\mid c^{5}=d^2=1,dcd=c^{-1}\rangle$ with the fixed listing $H=(h_{5j+i+1})=a^ib^j$ for $i\in[0\zint 4]$ and $j\in[0\zint 1]$. Let $H'=\vct{1}$ and $P'=\vct{1}$. Let $v=\sum_{i=1}^{20}{\alpha_{g_i}}g_i\in RG$. If $\Omega_{1}^{20}(v)$ is the composite $(G,H)$-matrix of $v\in RG$ with respect to $H'$ and $P'$, then
		\begin{align*}
			\Omega_{1}^{20}(v)=\begin{pmatrix}
				A_1 & B_1 & C_1 & D_1\\
				B_1^T & A_1^T & D_1^T & C_1^T\\
				C_2 & D_2 & A_2 & B_2\\
				D_2^T & C_2^T & B_2^T & A_2^T
			\end{pmatrix},
		\end{align*}
	where $A_1=\cir{\vct{v}_{1:5}}$, $B_1=\cir{\vct{v}_{6:10}}$, $C_1=\cir{\vct{v}_{11:15}}$, $D_1=\cir{\vct{v}_{16:20}}$, $A_2=\cir{\vct{v}_1,\vct{v}_{10:7}}$, $B_2=\cir{\vct{v}_{6:2}}$, $C_2=\cir{\vct{v}_{11},\vct{v}_{20:17}}$ and $D_2=\cir{\vct{v}_{16:12}}$.
	\end{definition}
	\begin{theorem}\label{theorem-Omega_20_1}
	Let $G=(I\,|\,\Omega_{1}^{20}(v))$ where $\Omega_{1}^{20}(v)$ is as defined in Definition \ref{definition-Omega_20_1}. Then $G$ is a generator matrix of a self-dual $[40,20]$ code over $R$ if and only if 
		\begin{empheq}[left=\empheqlbrace]{align*}
			A_1A_1^T+B_1B_1^T+C_1C_1^T+D_1D_1^T&=I_{5},\\
			A_2A_2^T+B_2B_2^T+C_2C_2^T+D_2D_2^T&=I_{5},\\
			A_1C_2^T+B_1D_2^T+C_1A_2^T+D_1B_2^T&=\vct{0},\\
			A_1D_2+B_1C_2+C_1B_2+D_1A_2&=\vct{0}.
		\end{empheq}
	\end{theorem}
	\begin{proof}
		 We know that $G$ is a generator matrix of a self-dual $[40,20]$ code over $R$ if and only if $\Omega_{1}^{20}(v)\Omega_{1}^{20}(v)^T=I_{20}$. We find that
			\begin{align*}
				\Omega_{1}^{20}(v)\Omega_{1}^{20}(v)^T=\begin{pmatrix}
					X_1 & \vct{0} & Y_1 & Y_2\\
					\vct{0} & X_1 & Y_2^T & Y_1^T\\
					Y_1^T & Y_2 & X_2 & \vct{0}\\
					Y_2^T & Y_1 & \vct{0} & X_2
				\end{pmatrix},
			\end{align*}
		where
			\begin{align*}
				X_1&=A_1A_1^T+B_1B_1^T+C_1C_1^T+D_1D_1^T,\\
				X_2&=A_2A_2^T+B_2B_2^T+C_2C_2^T+D_2D_2^T
			\end{align*}
		and
			\begin{align*}
				Y_1&=A_1C_2^T+B_1D_2^T+C_1A_2^T+D_1B_2^T,\\
				Y_2&=A_1D_2+B_1C_2+C_1B_2+D_1A_2.
			\end{align*}
		
		Clearly, $Y_i=\vct{0}$ if and only if $Y_i^T=\vct{0}$ for $i\in[1\zint 2]$. Thus, $\Omega_{1}^{20}(v)\Omega_{1}^{20}(v)^T=I_{20}$ if and only if
			\begin{empheq}[left=\empheqlbrace]{align*}
				X_1=X_2&=I_5,\\
				Y_1=Y_2&=\vct{0}.
			\end{empheq}
	\end{proof}
	
	\begin{definition}\label{definition-Omega_20_2}
	Let $G\cong C_5\times C_4\cong\langle a,b\mid a^5=b^4=1,ab=ba\rangle$ with the fixed listing $G=(g_{5j+i+1})=a^ib^j$ for $i\in[0\zint 4]$ and $j\in[0\zint 3]$. Let $H\cong D_{5}\cong\langle c,d\mid c^{5}=d^2=1,dcd=c^{-1}\rangle$ with the fixed listing $H=(h_{5j+i+1})=a^ib^j$ for $i\in[0\zint 4]$ and $j\in[0\zint 1]$. Let $H'=\vct{1}$ and $P'=\vct{1}$. Let $v=\sum_{i=1}^{20}{\alpha_{g_i}}g_i\in RG$. If $\Omega_{2}^{20}(v)$ is the composite $(G,H)$-matrix of $v\in RG$ with respect to $H'$ and $P'$, then
		\begin{align*}
			\Omega_{2}^{20}(v)=\begin{pmatrix}
				A & B & C & D\\
				B^T & A^T & D^T & C^T\\
				C & D & A & B\\
				D^T & C^T & B^T & A^T 	
			\end{pmatrix},
		\end{align*}
	where $A=\cir{\vct{v}_{1:5}}$, $B=\cir{\vct{v}_{6:10}}$, $C=\cir{\vct{v}_{11:15}}$ and $D=\cir{\vct{v}_{16:20}}$.
	\end{definition}
	\begin{theorem}\label{theorem-Omega_20_2}
	Let $G=(I\,|\,\Omega_{2}^{20}(v))$ where $\Omega_{2}^{20}(v)$ is as defined in Definition \ref{definition-Omega_20_2}. Then $G$ is a generator matrix of a self-dual $[40,20]$ code over $R$ if and only if 
		\begin{empheq}[left=\empheqlbrace]{align*}
			AA^T+BB^T+CC^T+DD^T&=I_5,\\
			AC^T+BD^T+CA^T+DB^T&=\vct{0}.
		\end{empheq}
	\end{theorem}
	\begin{proof}
		 We know that $G$ is a generator matrix of a self-dual $[40,20]$ code over $R$ if and only if $\Omega_{2}^{20}(v)\Omega_{2}^{20}(v)^T=I_{20}$. We find that
			\begin{align*}
				\Omega_{2}^{20}(v)\Omega_{2}^{20}(v)^T=\cir{X,\vct{0},Y,\vct{0}},
			\end{align*}
		where
			\begin{align*}
				X=AA^T+BB^T+CC^T+DD^T
			\end{align*}
		and
			\begin{align*}
				Y=AC^T+BD^T+CA^T+DB^T.
			\end{align*}
			
		Thus, $\Omega_{2}^{20}(v)\Omega_{2}^{20}(v)^T=I_{20}$ if and only if
			\begin{empheq}[left=\empheqlbrace]{align*}
				X&=I_5,\\
				Y&=\vct{0}.
			\end{empheq}
	\end{proof}

	\begin{definition}\label{definition-Omega_42_1}
	Let $G\cong D_{21}\cong\langle a,b\mid a^{21}=b^2=1,bab=a^{-1}\rangle$ with the fixed listing $G=(g_{21j+i+1})=a^ib^j$ for $i\in[0\zint 20]$ and $j\in[0\zint 1]$. Let $H\cong C_7\times C_3\cong\langle c,d\mid c^7=d^3=1,cd=dc\rangle$ with the fixed listing $H=(h_{7j+i+1})=a^ib^j$ for $i\in[0\zint 6]$ and $j\in[0\zint 2]$. Let $H'=\vct{1}$ and $P'=\vct{1}$. Let $v=\sum_{i=1}^{42}{\alpha_{g_i}}g_i\in RG$. If $\Omega_{1}^{42}(v)$ is the composite $(G,H)$-matrix of $v\in RG$ with respect to $H'$ and $P'$, then
		\begin{align*}
			\Omega_{1}^{42}(v)=\begin{pmatrix}
				\CIR{A_1,A_2,A_3} & \CIR{B_1,B_2,B_3}\\
				\CIR{C_1,C_2,C_3} & \CIR{D_1,D_2,D_3}	
			\end{pmatrix},
		\end{align*}
	where $A_1=\cir{\vct{v}_{1:7}}$, $A_2=\cir{\vct{v}_{8:14}}$, $A_3=\cir{\vct{v}_{15:21}}$, $B_1=\cir{\vct{v}_{22:28}}$, $B_2=\cir{\vct{v}_{29:35}}$, $B_3=\cir{\vct{v}_{36:42}}$, $C_1=\cir{\vct{v}_{22},\vct{v}_{42:37}}$, $C_2=\cir{\vct{v}_{36:30}}$, $C_3=\cir{\vct{v}_{29:23}}$, $D_1=\cir{\vct{v}_1,\vct{v}_{21:16}}$, $D_2=\cir{\vct{v}_{15:9}}$ and $D_2=\cir{\vct{v}_{8:2}}$.
	\end{definition}
	\begin{theorem}\label{theorem-Omega_42_1}
	Let $G=(I\,|\,\Omega_{1}^{42}(v))$ where $\Omega_{1}^{42}(v)$ is as defined in Definition \ref{definition-Omega_42_1}. Then $G$ is a generator matrix of a self-dual $[84,42]$ code over $R$ if and only if 
		\begin{empheq}[left=\empheqlbrace]{align*}
			A_1A_1^T+A_2A_2^T+A_3A_3^T+B_1B_1^T+B_2B_2^T+B_3B_3^T&=I_7,\\
			C_1C_1^T+C_2C_2^T+C_3C_3^T+D_1D_1^T+D_2D_2^T+D_3D_3^T&=I_7,\\
			A_1A_3^T+A_2A_1^T+A_3A_2^T+B_1B_3^T+B_2B_1^T+B_3B_2^T&=\vct{0},\\
			A_1C_1^T+A_2C_2^T+A_3C_3^T+B_1D_1^T+B_2D_2^T+B_3D_3^T&=\vct{0},\\
			A_1C_3^T+A_2C_1^T+A_3C_2^T+B_1D_3^T+B_2D_1^T+B_3D_2^T&=\vct{0},\\
			A_1C_2^T+A_2C_3^T+A_3C_1^T+B_1D_2^T+B_2D_3^T+B_3D_1^T&=\vct{0},\\
			C_1C_3^T+C_2C_1^T+C_3C_2^T+D_1D_3^T+D_2D_1^T+D_3D_2^T&=\vct{0}.
		\end{empheq}
	\end{theorem}
	\begin{proof}
		We know that $G$ is a generator matrix of a self-dual $[84,42]$ code over $R$ if and only if $\Omega_{1}^{42}(v)\Omega_{1}^{42}(v)^T=I_{42}$. We find that
			\begin{align*}
				\Omega_{1}^{42}(v)\Omega_{1}^{42}(v)^T=\begin{pmatrix}
					\CIR{X_1,Y_1,Y_1^T} & \CIR{Y_2,Y_3,Y_4}\\
					\CIR{Y_2^T,Y_4^T,Y_3^T} & \CIR{X_2,Y_5,Y_5^T}
				\end{pmatrix},
			\end{align*}
		where
			\begin{align*}
				X_1&=A_1A_1^T+A_2A_2^T+A_3A_3^T+B_1B_1^T+B_2B_2^T+B_3B_3^T,\\
				X_2&=C_1C_1^T+C_2C_2^T+C_3C_3^T+D_1D_1^T+D_2D_2^T+D_3D_3^T
			\end{align*}
		and
			\begin{align*}
				Y_1&=A_1A_3^T+A_2A_1^T+A_3A_2^T+B_1B_3^T+B_2B_1^T+B_3B_2^T,\\
				Y_2&=A_1C_1^T+A_2C_2^T+A_3C_3^T+B_1D_1^T+B_2D_2^T+B_3D_3^T,\\
				Y_3&=A_1C_3^T+A_2C_1^T+A_3C_2^T+B_1D_3^T+B_2D_1^T+B_3D_2^T,\\
				Y_4&=A_1C_2^T+A_2C_3^T+A_3C_1^T+B_1D_2^T+B_2D_3^T+B_3D_1^T,\\
				Y_5&=C_1C_3^T+C_2C_1^T+C_3C_2^T+D_1D_3^T+D_2D_1^T+D_3D_2^T.
			\end{align*}
			
		Clearly, $Y_i=\vct{0}$ if and only if $Y_i^T=\vct{0}$ for $i\in[1\zint 5]$. Thus, $\Omega_{1}^{42}(v)\Omega_{1}^{42}(v)^T=I_{42}$ if and only if
			\begin{empheq}[left=\empheqlbrace]{align*}
				X_1=X_2&=I_7,\\
				Y_1=Y_2=Y_3=Y_4=Y_5&=\vct{0}.
			\end{empheq}
	\end{proof}

	\begin{definition}\label{definition-Omega_42_2}
	Let $G\cong D_{21}\cong\langle a,b\mid a^{21}=b^2=1,bab=a^{-1}\rangle$ with the fixed listing $G=(g_{21j+i+1})=a^ib^j$ for $i\in[0\zint 20]$ and $j\in[0\zint 1]$. Let $H\cong C_3\times C_7\cong\langle c,d\mid c^3=b^7=1,cd=dc\rangle$ with the fixed listing $H=(h_{3j+i+1})=c^id^j$ for $i\in[0\zint 2]$ and $j\in[0\zint 6]$. Let $H'=\vct{1}$ and $P'=\vct{1}$. Let $v=\sum_{i=1}^{42}{\alpha_{g_i}}g_i\in RG$. If $\Omega_{2}^{42}(v)$ is the composite $(G,H)$-matrix of $v\in RG$ with respect to $H'$ and $P'$, then
		\begin{align*}
			\Omega_{2}^{42}(v)=\begin{pmatrix}
				A_1 & A_2 & A_3 & B_1 & B_2 & B_3\\
				A_3^{\star} & A_1 & A_2 & B_3^{\star} & B_1 & B_2\\
				A_2^{\star} & A_3^{\star} & A_1 & B_2^{\star} & B_3^{\star} & B_1\\
				C_1 & C_2 & C_3 & D_1 & D_2 & D_3\\
				C_3^{\star} & C_1 & C_2 & D_3^{\star} & D_1 & D_2\\
				C_2^{\star} & C_3^{\star} & C_1& D_2^{\star} & D_3^{\star} & D_1
			\end{pmatrix},
		\end{align*}
	where $A_1=\cir{\vct{v}_{1:7}}$, $A_2=\cir{\vct{v}_{8:14}}$, $A_3=\cir{\vct{v}_{15:21}}$, $B_1=\cir{\vct{v}_{22:28}}$, $B_2=\cir{\vct{v}_{29:35}}$, $B_3=\cir{\vct{v}_{36:42}}$, $C_1=\cir{\vct{v}_{22},\vct{v}_{42:37}}$, $C_2=\cir{\vct{v}_{36:30}}$, $C_3=\cir{\vct{v}_{29:23}}$, $D_1=\cir{\vct{v}_{1},\vct{v}_{21:16}}$, $D_2=\cir{\vct{v}_{15:9}}$, $D_2=\cir{\vct{v}_{8:2}}$ and $\star$ is the transformation defined in Proposition \ref{proposition-star}.
	\end{definition}
	\begin{theorem}\label{theorem-Omega_42_2}
	Let $G=(I\,|\,\Omega_{2}^{42}(v))$ where $\Omega_{2}^{42}(v)$ is as defined in Definition \ref{definition-Omega_42_2}. Then $G$ is a generator matrix of a self-dual $[84,42]$ code over $R$ if and only if 
		\begin{empheq}[left=\empheqlbrace]{align*}
			A_{1}A_{1}^T+A_{2}A_{2}^T+A_{3}A_{3}^T+B_{1}B_{1}^T+B_{2}B_{2}^T+B_{3}B_{3}^T&=I_7,\\
			C_{1}C_{1}^T+C_{2}C_{2}^T+C_{3}C_{3}^T+D_{1}D_{1}^T+D_{2}D_{2}^T+D_{3}D_{3}^T&=I_7,\\
			A_{2}A_{1}^T+A_{3}A_{2}^T+A_{1}A_{3}^{\star T}+B_{2}B_{1}^T+B_{3}B_{2}^T+B_{1}B_{3}^{\star T}&=\vct{0},\\
			A_{1}C_{1}^T+A_{2}C_{2}^T+A_{3}C_{3}^T+B_{1}D_{1}^T+B_{2}D_{2}^T+B_{3}D_{3}^T&=\vct{0},\\
			A_{2}C_{1}^T+A_{3}C_{2}^T+A_{1}C_{3}^{\star T}+B_{2}D_{1}^T+B_{3}D_{2}^T+B_{1}D_{3}^{\star T}&=\vct{0},\\
			A_{3}C_{1}^T+A_{1}C_{2}^{\star T}+A_{2}C_{3}^{\star T}+B_{3}D_{1}^T+B_{1}D_{2}^{\star T}+B_{2}D_{3}^{\star T}&=\vct{0},\\
			C_{2}C_{1}^T+C_{3}C_{2}^T+C_{1}C_{3}^{\star T}+D_{2}D_{1}^T+D_{3}D_{2}^T+D_{1}D_{3}^{\star T}&=\vct{0}.
		\end{empheq}
	\end{theorem}
	\begin{proof}
		We know that $G$ is a generator matrix of a self-dual $[84,42]$ code over $R$ if and only if $\Omega_{2}^{42}(v)\Omega_{2}^{42}(v)^T=I_{42}$. Using Lemmas \ref{lemma-star} and \ref{lemma-starcirc}, we find that
			\begin{align*}
				\Omega_{2}^{42}(v)\Omega_{2}^{42}(v)^T=\begin{pmatrix}
					X_1 & Y_1 & Y_1^{\star T} & Y_2 & Y_3 & Y_4\\
					Y_1^T & X_1 & Y_1 & Y_4^{\star} & Y_2 & Y_3\\
					Y_1^{\star} & Y_1^T & X_1 & Y_3^{\star} & Y_4^{\star} & Y_2\\
					Y_2^T & Y_4^{\star T} & Y_3^{\star T} & X_2 & Y_5 & Y_5^{\star T}\\
					Y_3^T & Y_2^T & Y_4^{\star T} & Y_5^T & X_2 & Y_5\\
					Y_4^T & Y_3^T & Y_2^T & Y_5^{\star} & Y_5^T & X_2
				\end{pmatrix},
			\end{align*}
		where
			\begin{align*}
				X_1&=A_{1}A_{1}^T+A_{2}A_{2}^T+A_{3}A_{3}^T+B_{1}B_{1}^T+B_{2}B_{2}^T+B_{3}B_{3}^T,\\
				X_2&=C_{1}C_{1}^T+C_{2}C_{2}^T+C_{3}C_{3}^T+D_{1}D_{1}^T+D_{2}D_{2}^T+D_{3}D_{3}^T\\
			\end{align*}	
		and
			\begin{align*}
				Y_1&=A_{2}A_{1}^T+A_{3}A_{2}^T+A_{1}A_{3}^{\star T}+B_{2}B_{1}^T+B_{3}B_{2}^T+B_{1}B_{3}^{\star T},\\
				Y_2&=A_{1}C_{1}^T+A_{2}C_{2}^T+A_{3}C_{3}^T+B_{1}D_{1}^T+B_{2}D_{2}^T+B_{3}D_{3}^T,\\
				Y_3&=A_{2}C_{1}^T+A_{3}C_{2}^T+A_{1}C_{3}^{\star T}+B_{2}D_{1}^T+B_{3}D_{2}^T+B_{1}D_{3}^{\star T},\\
				Y_4&=A_{3}C_{1}^T+A_{1}C_{2}^{\star T}+A_{2}C_{3}^{\star T}+B_{3}D_{1}^T+B_{1}D_{2}^{\star T}+B_{2}D_{3}^{\star T},\\
				Y_5&=C_{2}C_{1}^T+C_{3}C_{2}^T+C_{1}C_{3}^{\star T}+D_{2}D_{1}^T+D_{3}D_{2}^T+D_{1}D_{3}^{\star T}.
			\end{align*}
			
		Clearly, $Y_i=\vct{0}$ if and only if $Y_i^T=\vct{0}$, $Y_i^{\star}=\vct{0}$ and $Y_i^{\star T}=\vct{0}$ for $i\in[1\zint 5]$. Thus, $\Omega_{2}^{42}(v)\Omega_{2}^{42}(v)^T=I_{42}$ if and only if
			\begin{empheq}[left=\empheqlbrace]{align*}
				X_1=X_2&=I_5,\\
				Y_1=Y_2=Y_3=Y_4=Y_5&=\vct{0}.
			\end{empheq}
	\end{proof}

	\begin{definition}\label{definition-Omega_24_1}
	Let $G\cong C_{12}\times C_2\cong\langle a,b\mid a^{12}=b^2=1,ab=ba\rangle$ with the fixed listing $G=(g_{12j+i+1})=a^ib^j$ for $i\in[0\zint 11]$ and $j\in[0\zint 1]$. Let $H\cong D_3\cong\langle c,d\mid c^3=d^2=1,dcd=c^{-1}\rangle$ with the fixed listing $H=(h_{3j+i+1})=c^id^j$ for $i\in[0\zint 2]$ and $j\in[0\zint 1]$. Let $H'=\vct{1}$ and $P'=\vct{1}$. Let $v=\sum_{i=1}^{24}{\alpha_{g_i}}g_i\in RG$. If $\Omega_{1}^{24}(v)$ is the composite $(G,H)$-matrix of $v\in RG$ with respect to $H'$ and $P'$, then
		\begin{align*}
			\Omega_{1}^{24}(v)=I_2\otimes\CIR{\tilde{A},\tilde{B}}+J_2\otimes\CIR{\tilde{C},\tilde{D}},
		\end{align*}
	where $\otimes$, $I_2$ and $J_2$ denote the Kronecker product, $2\times 2$ identity matrix and $2\times 2$ exchange matrix, respectively and
		\begin{align*}
			&\tilde{A}=\begin{pmatrix}
			A_1 & A_2\\
			A_2^T & A_1^T
			\end{pmatrix},\quad
			\tilde{B}=\begin{pmatrix}
			B_1 & B_2\\
			B_2^T & B_1^T
			\end{pmatrix},\\
			&\tilde{C}=\begin{pmatrix}
			C_1 & C_2\\
			C_2^T & C_1^T
			\end{pmatrix},\quad
			\tilde{D}=\begin{pmatrix}
			D_1 & D_2\\
			D_2^T & D_1^T
			\end{pmatrix},
		\end{align*}
	where $A_1=\cir{\vct{v}_{1:3}}$, $A_2=\cir{\vct{v}_{4:6}}$, $B_1=\cir{\vct{v}_{7:9}}$, $B_2=\cir{\vct{v}_{10:12}}$, $C_1=\cir{\vct{v}_{13:15}}$, $C_2=\cir{\vct{v}_{16:18}}$, $D_1=\cir{\vct{v}_{19:21}}$ and $D_2=\cir{\vct{v}_{22:24}}$.
	\end{definition}
	\begin{theorem}\label{theorem-Omega_24_1}
	Let $G=(I\,|\,\Omega_{1}^{24}(v))$ where $\Omega_{1}^{24}(v)$ is as defined in Definition \ref{definition-Omega_24_1}. Then $G$ is a generator matrix of a self-dual $[48,24]$ code over $R$ if and only if 
		\begin{empheq}[left=\empheqlbrace]{align*}
			A_1A_1^T+A_2A_2^T+B_1B_1^T+B_2B_2^T+C_1C_1^T+C_2C_2^T+D_1D_1^T+D_2D_2^T&=I_3,\\
		    A_1B_1^T+A_2B_2^T+B_1A_1^T+B_2A_2^T+C_1D_1^T+C_2D_2^T+D_1C_1^T+D_2C_2^T&=\vct{0},\\
			A_1C_1^T+A_2C_2^T+B_1D_1^T+B_2D_2^T+C_1A_1^T+C_2A_2^T+D_1B_1^T+D_2B_2^T&=\vct{0},\\
			A_1D_1^T+A_2D_2^T+B_1C_1^T+B_2C_2^T+C_1B_1^T+C_2B_2^T+D_1A_1^T+D_2A_2^T&=\vct{0}.
		\end{empheq}
	\end{theorem}
	\begin{proof}
	    We know that $G$ is a generator matrix of a self-dual $[48,24]$ code over $R$ if and only if $\Omega_{1}^{24}(v)\Omega_{1}^{24}(v)^T=I_{24}$. We find that
			\begin{align*}
				\Omega_{1}^{24}(v)\Omega_{1}^{24}(v)^T=I_2\otimes\CIR{\tilde{X},\tilde{Y}_1}+J_2\otimes\CIR{\tilde{Y}_2,\tilde{Y}_3},
			\end{align*}
		where
			\begin{alignat*}{2}
				&\tilde{X}=\begin{pmatrix}
				X & \vct{0}\\
				\vct{0} & X
				\end{pmatrix},\quad
				&&\tilde{Y}_1=\begin{pmatrix}
				Y_1 & \vct{0}\\
				\vct{0} & Y_1
				\end{pmatrix},\\
				&\tilde{Y}_2=\begin{pmatrix}
				Y_2 & \vct{0}\\
				\vct{0} & Y_2
				\end{pmatrix},\quad
				&&\tilde{Y}_3=\begin{pmatrix}
				Y_3 & \vct{0}\\
				\vct{0} & Y_3
				\end{pmatrix}
			\end{alignat*}
		with
			\begin{align*}
				X&=A_1A_1^T+A_2A_2^T+B_1B_1^T+B_2B_2^T+C_1C_1^T+C_2C_2^T+D_1D_1^T+D_2D_2^T
			\end{align*}	
		and
			\begin{align*}
				Y_1&=A_1B_1^T+A_2B_2^T+B_1A_1^T+B_2A_2^T+C_1D_1^T+C_2D_2^T+D_1C_1^T+D_2C_2^T,\\
				Y_2&=A_1C_1^T+A_2C_2^T+B_1D_1^T+B_2D_2^T+C_1A_1^T+C_2A_2^T+D_1B_1^T+D_2B_2^T,\\
				Y_3&=A_1D_1^T+A_2D_2^T+B_1C_1^T+B_2C_2^T+C_1B_1^T+C_2B_2^T+D_1A_1^T+D_2A_2^T.
			\end{align*}
			
        Thus, $\Omega_{1}^{24}(v)\Omega_{1}^{24}(v)^T=I_{24}$ if and only if
			\begin{empheq}[left=\empheqlbrace]{align*}
				X_1=X_2&=I_3,\\
				Y_1=Y_2=Y_3&=\vct{0}.
			\end{empheq}
	\end{proof}

	\begin{definition}\label{definition-Omega_24_2}
	Let $G\cong D_{12}\cong\langle a,b\mid a^{12}=b^2=1,bab=a^{-1}\rangle$ with the fixed listing $G=(g_{12j+i+1})=a^ib^j$ for $i\in[0\zint 11]$ and $j\in[0\zint 1]$. Let $H\cong C_{2\cdot 6}\cong\langle c\mid c^{2\cdot 6}=1\rangle$ with the fixed listing $H=(h_{6j+i+1})=c^{2i+j}$ for $i\in[0\zint 5]$ and $j\in[0\zint 1]$. Let $H'=\vct{1}$ and $P'=\vct{1}$. Let $v=\sum_{i=1}^{24}{\alpha_{g_i}}g_i\in RG$. If $\Omega_{2}^{24}(v)$ is the composite $(G,H)$-matrix of $v\in RG$ with respect to $H'$ and $P'$, then
		\begin{align*}
			\Omega_{2}^{24}(v)=\begin{pmatrix}
				A_1 & A_2 & B_1 & B_2\\
				A_2^{\star} & A_1 & B_2^{\star} & B_1\\
				C_1 & C_2 & D_1 & D_2\\
				C_2^{\star} & C_1 & D_2^{\star} & D_1	
			\end{pmatrix},
		\end{align*}
	where $A_1=\cir{\vct{v}_{1:6}}$, $A_2=\cir{\vct{v}_{7:12}}$, $B_1=\cir{\vct{v}_{13:18}}$, $B_2=\cir{\vct{v}_{19:24}}$, $C_1=\cir{\vct{v}_{13},\vct{v}_{24:20}}$, $C_2=\cir{\vct{v}_{19:14}}$, $D_1=\cir{\vct{v}_{1},\vct{v}_{12:8}}$, $D_2=\cir{\vct{v}_{7:2}}$ and $\star$ is the transformation defined in Proposition \ref{proposition-star}.
	\end{definition}
	\begin{theorem}\label{theorem-Omega_24_2}
	Let $G=(I\,|\,\Omega_{2}^{24}(v))$ where $\Omega_{2}^{24}(v)$ is as defined in Definition \ref{definition-Omega_24_2}. Then $G$ is a generator matrix of a self-dual $[48,24]$ code over $R$ if and only if 
		\begin{empheq}[left=\empheqlbrace]{align*}
			A_1A_1^T+A_2A_2^T+B_1B_1^T+B_2B_2^T&=I_6,\\
			C_1C_1^T+C_2C_2^T+D_1D_1^T+D_2D_2^T&=I_6,\\
			A_1A_2^{\star T}+A_2A_1^T+B_1B_2^{\star T}+B_2B_1^T&=\vct{0},\\
			A_1C_1^T+A_2C_2^T+B_1D_1^T+B_2D_2^T&=\vct{0},\\
			A_1C_2^{\star T}+A_2C_1^T+B_1D_2^{\star T}+B_2D_1^T&=\vct{0},\\
			C_1C_2^{\star T}+C_2C_1^T+D_1D_2^{\star T}+D_2D_1^T&=\vct{0}.
		\end{empheq}
	\end{theorem}
	\begin{proof}
		We know that $G$ is a generator matrix of a self-dual $[48,24]$ code over $R$ if and only if $\Omega_{2}^{24}(v)\Omega_{2}^{24}(v)^T=I_{24}$. Using Lemmas \ref{lemma-star} and \ref{lemma-starcirc}, we find that
			\begin{align*}
				\Omega_{2}^{24}(v)\Omega_{2}^{24}(v)^T=\begin{pmatrix}
					X_1 & Y_1 & Y_2 & Y_3\\
					Y_1^T & X_1 & Y_3^{\star} & Y_2\\
					Y_2^T & Y_3^{\star T} & X_2 &  Y_4\\
					Y_3^T & Y_2^T & Y_4^T & X_2
				\end{pmatrix},
			\end{align*}
		where
			\begin{align*}
				X_1&=A_1A_1^T+A_2A_2^T+B_1B_1^T+B_2B_2^T,\\
				X_2&=C_1C_1^T+C_2C_2^T+D_1D_1^T+D_2D_2^T
			\end{align*}	
		and
			\begin{align*}
				Y_1&=A_1A_2^{\star T}+A_2A_1^T+B_1B_2^{\star T}+B_2B_1^T,\\
				Y_2&=A_1C_1^T+A_2C_2^T+B_1D_1^T+B_2D_2^T,\\
				Y_3&=A_1C_2^{\star T}+A_2C_1^T+B_1D_2^{\star T}+B_2D_1^T,\\
				Y_4&=C_1C_2^{\star T}+C_2C_1^T+D_1D_2^{\star T}+D_2D_1^T.
			\end{align*}
			
		Clearly, $Y_i=\vct{0}$ if and only if $Y_i^T=\vct{0}$, $Y_i^{\star}=\vct{0}$ and $Y_i^{\star T}=\vct{0}$ for $i\in[1\zint 4]$. Thus, $\Omega_{2}^{24}(v)\Omega_{2}^{24}(v)^T=I_{24}$ if and only if
			\begin{empheq}[left=\empheqlbrace]{align*}
				X_1=X_2&=I_6,\\
				Y_1=Y_2=Y_3=Y_4&=\vct{0}.
			\end{empheq}
	\end{proof}

	\begin{definition}\label{definition-Omega_24_3}
	Let $G\cong D_{12}\cong\langle a,b\mid a^{12}=b^2=1,bab=a^{-1}\rangle$ with the fixed listing $G=(g_{12j+i+1})=a^ib^j$ for $i\in[0\zint 11]$ and $j\in[0\zint 1]$. Let $H\cong D_6\cong\langle c,d\mid c^6=d^2=1,dcd=c^{-1}\rangle$ with the fixed listing $H=(h_{6j+i+1})=c^id^j$ for $i\in[0\zint 5]$ and $j\in[0\zint 1]$. Let $H'=\vct{1}$ and $P'=\vct{1}$. Let $v=\sum_{i=1}^{24}{\alpha_{g_i}}g_i\in RG$. If $\Omega_{3}^{24}(v)$ is the composite $(G,H)$-matrix of $v\in RG$ with respect to $H'$ and $P'$, then
		\begin{align*}
			\Omega_{3}^{24}(v)=\begin{pmatrix}
				A_1 & A_2 & B_1 & B_2\\
				A_2^T & A_1^T & B_2^T & B_1^T\\
				C_1 & C_2 & D_1 & D_2\\
				C_2^T & C_1^T & D_2^T & D_1^T
			\end{pmatrix},
		\end{align*}
	where $A_1=\cir{\vct{v}_{1:6}}$, $A_2=\cir{\vct{v}_{7:12}}$, $B_1=\cir{\vct{v}_{13:18}}$, $B_2=\cir{\vct{v}_{19:24}}$, $C_1=\cir{\vct{v}_{13},\vct{v}_{24:20}}$, $C_2=\cir{\vct{v}_{19:14}}$, $D_1=\cir{\vct{v}_{1},\vct{v}_{12:8}}$ and $D_2=\cir{\vct{v}_{7:2}}$.
	\end{definition}
	\begin{theorem}\label{theorem-Omega_24_3}
	Let $G=(I\,|\,\Omega_{3}^{24}(v))$ where $\Omega_{3}^{24}(v)$ is as defined in Definition \ref{definition-Omega_24_3}. Then $G$ is a generator matrix of a self-dual $[48,24]$ code over $R$ if and only if 
		\begin{empheq}[left=\empheqlbrace]{align*}
			A_1A_1^T+A_2A_2^T+B_1B_1^T+B_2B_2^T&=I_6,\\
			C_1C_1^T+C_2C_2^T+D_1D_1^T+D_2D_2^T&=I_6,\\
			A_1C_1^T+A_2C_2^T+B_1D_1^T+B_2D_2^T&=\vct{0},\\
			A_1C_2+A_2C_1+B_1D_2+B_2D_1&=\vct{0}.
		\end{empheq}
	\end{theorem}
	\begin{proof}
		We know that $G$ is a generator matrix of a self-dual $[48,24]$ code over $R$ if and only if $\Omega_{3}^{24}(v)\Omega_{3}^{24}(v)^T=I_{24}$. We find that
			\begin{align*}
				\Omega_{3}^{24}(v)\Omega_{3}^{24}(v)^T=\begin{pmatrix}
				    X_1 & \vct{0} & Y_1 & Y_2\\
				    \vct{0} & X_1 & Y_2^T & Y_1^T\\
				    Y_1^T & Y_2 & X_2 & \vct{0}\\
				    Y_2^T & Y_1 & \vct{0} & X_2
				\end{pmatrix},
			\end{align*}
		where
			\begin{align*}
				X_1&=A_1A_1^T+A_2A_2^T+B_1B_1^T+B_2B_2^T,\\
				X_2&=C_1C_1^T+C_2C_2^T+D_1D_1^T+D_2D_2^T
			\end{align*}	
		and
			\begin{align*}
				Y_1&=A_1C_1^T+A_2C_2^T+B_1D_1^T+B_2D_2^T,\\
				Y_2&=A_1C_2+A_2C_1+B_1D_2+B_2D_1.
			\end{align*}
			
		Clearly, $Y_i=\vct{0}$ if and only if $Y_i^T=\vct{0}$ for $i\in[1\zint 2]$. Thus, $\Omega_{3}^{24}(v)\Omega_{3}^{24}(v)^T=I_{24}$ if and only if
			\begin{empheq}[left=\empheqlbrace]{align*}
				X_1=X_2&=I_6,\\
				Y_1=Y_2&=\vct{0}.
			\end{empheq}
	\end{proof}

\section{Results}

In this section, we apply the theorems given in the previous section to obtain many new best known binary self-dual codes. In particular, we obtain 28 singly-even $[80,40,14]$ codes, 107 $[84,42,14]$ codes, 105 singly-even $[96,48,16]$ codes and 121 doubly-even $[96,48,16]$ codes. 
	 
We search for these codes using MATLAB and determine their properties using Q-extension \cite{Q-extension} and Magma \cite{magma}. In MATLAB, we employ an algorithm which randomly searches for the construction parameters that satisfy the necessary and sufficient conditions stated in the corresponding theorem. For such parameters, we then build the corresponding binary generator matrices and print them to text files. We then use Q-extension to read these text files and determine the minimum distance and partial weight enumerator of each corresponding code. Furthermore, we determine the automorphism group order of each code using Magma. A database of generator matrices of the new codes is given online at \cite{gmd}. The database is partitioned into text files (interpretable by Q-extension) corresponding to each code type. In these files, specific properties of the codes including the construction parameters, weight enumerator parameter values and automorphism group order are formatted as comments above the generator matrices. Partial weight enumerators of the codes are also formatted as comments below the generator matrices. Table \ref{table-notation} gives the quaternary notation system we use to represent elements of $\FF_2+u\FF_2$ and $\FF_4$.
	
	\begin{table}[h!]
	\caption{Quaternary notation system for elements of $\FF_2+u\FF_2$ and $\FF_4$.}\label{table-notation}
	\centering
	\begin{adjustbox}{max width=\textwidth}
	\footnotesize
	\begin{tabular}{ccc}\midrule
	$\FF_2+u\FF_2$ & $\FF_4$ & Symbol\\\midrule
	$0$ & $0$ & \texttt{0}\\
	$1$ & $1$ & \texttt{1}\\
	$u$ & $w$ & \texttt{2}\\
	$1+u$ & $1+w$ & \texttt{3}\\\midrule
	\end{tabular}
	\end{adjustbox}
	\end{table}

\subsection{New Self-Dual Codes of Length 80}
The weight enumerator of a singly-even binary self-dual $[80,40,14]$ code is given in \cite{R-085} as
	\begin{align*}
	W_{80}=1+(3200+4\alpha)x^{14}+(47645-8\alpha+256\beta)x^{16}+\cdots,
	\end{align*}
where $\alpha,\beta\in\ZZ$. Previously known $(\alpha,\beta)$ values for weight enumerator $W_{80}$ can be found online at \cite{wepd} (see \cite{R-058,R-085,R-080,R-067,R-096,mres,amr1}).

We obtain 28 new best known singly-even binary self-dual codes of length 80 which have weight enumerator $W_{80}$ for
	\begin{enumerate}[label=]
	\item $\beta=0$ and $\alpha\in\{-z:z=65,\lb80,\lb120,\lb125,\lb130,\lb135,\lb140,\lb145,\lb150,\lb155,\lb165,\lb175,\lb190,\lb195,\lb205,\lb210,\lb215,\lb230,\lb235,\lb250,\lb270,\lb275,\lb280,\lb360\}$;
	\item $\beta=10$ and $\alpha\in\{-2z:z=130,\lb150,\lb160,\lb185\}$.
	\end{enumerate}
	
Of the 28 new codes, 19 are constructed by applying Theorem \ref{theorem-Omega_20_1} over $\FF_4$ (Table \ref{table-80-Omega_20_1-F4}); 4 are constructed by applying Theorem \ref{theorem-Omega_20_2} over $\FF_2+u\FF_2$ (Table \ref{table-80-Omega_20_2-F2u}) and 5 are constructed by applying Theorem \ref{theorem-Omega_20_2} over $\FF_4$ (Table \ref{table-80-Omega_20_2-F4}).

	\begin{table}[h!]
	\caption{New singly-even binary self-dual $[80,40,14]$ codes from Theorem \ref{theorem-Omega_20_1} over $\FF_4$.}\label{table-80-Omega_20_1-F4}
	\centering
	\begin{adjustbox}{max width=\textwidth}
	\footnotesize
	\begin{tabular}{ccccc}\midrule
	$\mathcal{C}_{80,i}$ & $\vct{v}$ & $\alpha$ & $\beta$ & $|\aut{\mathcal{C}_{80,i}}|$\\\midrule
	1 & \texttt{(31223333300320201200)} & $-275$ & $0$ & $2^{2}\cdot 5$\\
	2 & \texttt{(13111130203000233223)} & $-270$ & $0$ & $2^{2}\cdot 5$\\
	3 & \texttt{(00302012331122313103)} & $-250$ & $0$ & $2^{2}\cdot 5$\\
	4 & \texttt{(01332030221111113310)} & $-235$ & $0$ & $2^{2}\cdot 5$\\
	5 & \texttt{(23320213130330103221)} & $-230$ & $0$ & $2^{3}\cdot 5$\\
	6 & \texttt{(22011233231033013100)} & $-210$ & $0$ & $2^{2}\cdot 5$\\
	7 & \texttt{(11333122033223331212)} & $-205$ & $0$ & $2^{2}\cdot 5$\\
	8 & \texttt{(02222010112332220213)} & $-195$ & $0$ & $2^{2}\cdot 5$\\
	9 & \texttt{(30333313000233100021)} & $-190$ & $0$ & $2^{2}\cdot 5$\\
	10 & \texttt{(00111023231213313321)} & $-175$ & $0$ & $2^{2}\cdot 5$\\
	11 & \texttt{(22310231030332003032)} & $-165$ & $0$ & $2^{2}\cdot 5$\\
	12 & \texttt{(02002030232203221313)} & $-155$ & $0$ & $2^{2}\cdot 5$\\
	13 & \texttt{(03121003232002123332)} & $-150$ & $0$ & $2^{2}\cdot 5$\\
	14 & \texttt{(23200233120101002302)} & $-145$ & $0$ & $2^{2}\cdot 5$\\
	15 & \texttt{(22133232333121133232)} & $-140$ & $0$ & $2^{3}\cdot 5$\\
	16 & \texttt{(31031330230000203122)} & $-135$ & $0$ & $2^{2}\cdot 5$\\
	17 & \texttt{(01103022122003122122)} & $-130$ & $0$ & $2^{2}\cdot 5$\\
	18 & \texttt{(22010203131000112213)} & $-65$ & $0$ & $2^{2}\cdot 5$\\
	19 & \texttt{(02330020210322001303)} & $-260$ & $10$ & $2^{3}\cdot 5$\\\midrule
	\end{tabular}
	\end{adjustbox}
	\end{table}

	\begin{table}[h!]
	\caption{New singly-even binary self-dual $[80,40,14]$ codes from Theorem \ref{theorem-Omega_20_2} over $\FF_2+u\FF_2$.}\label{table-80-Omega_20_2-F2u}
	\centering
	\begin{adjustbox}{max width=\textwidth}
	\footnotesize
	\begin{tabular}{ccccc}\midrule
	$\mathcal{C}_{80,i}$ & $\vct{v}$ & $\alpha$ & $\beta$ & $|\aut{\mathcal{C}_{80,i}}|$\\\midrule
	20 & \texttt{(12222331200322021203)} & $-280$ & $0$ & $2^{3}\cdot 5$\\
	21 & \texttt{(23330310032021331010)} & $-120$ & $0$ & $2^{3}\cdot 5$\\
	22 & \texttt{(30320122023203322322)} & $-80$ & $0$ & $2^{3}\cdot 5$\\
	23 & \texttt{(21222311321120112303)} & $-320$ & $10$ & $2^{4}\cdot 5$\\\midrule
	\end{tabular}
	\end{adjustbox}
	\end{table}
	
	\begin{table}[h!]
	\caption{New singly-even binary self-dual $[80,40,14]$ codes from Theorem \ref{theorem-Omega_20_2} over $\FF_4$.}\label{table-80-Omega_20_2-F4}
	\centering
	\begin{adjustbox}{max width=\textwidth}
	\footnotesize
	\begin{tabular}{ccccc}\midrule
	$\mathcal{C}_{80,i}$ & $\vct{v}$ & $\alpha$ & $\beta$ & $|\aut{\mathcal{C}_{80,i}}|$\\\midrule
	24 & \texttt{(31211223330300232332)} & $-360$ & $0$ & $2^{2}\cdot 5$\\
	25 & \texttt{(10201301032322330300)} & $-215$ & $0$ & $2^{2}\cdot 5$\\
	26 & \texttt{(01021003132222203113)} & $-125$ & $0$ & $2^{2}\cdot 5$\\
	27 & \texttt{(31003000101110232322)} & $-370$ & $10$ & $2^{2}\cdot 5$\\
	28 & \texttt{(11210213102203230313)} & $-300$ & $10$ & $2^{2}\cdot 5$\\\midrule
	\end{tabular}
	\end{adjustbox}
	\end{table}

\subsection{New Self-Dual Codes of Length 84}

The possible weight enumerators of a binary self-dual $[84,42,14]$ code are given in \cite{R-044,R-085} as
	\begin{align*}
	W_{84,1}&=1+(4080-\alpha)x^{14}+39524x^{16}\\&\quad+(247264+14\alpha)x^{18}+\cdots,\\
	W_{84,2}&=1+(4080-\alpha)x^{14}+(28644+64\beta)x^{16}\\&\quad+(390368+14\alpha-384\beta)x^{18}+\cdots,\\
	W_{84,3}&=1+(4080-\alpha)x^{14}+(28644+64\beta)x^{16}\\&\quad+(394464+14\alpha-384\beta)x^{18}+\cdots,
	\end{align*}
where $\alpha,\beta\in\ZZ$. Previously known $(\alpha,\beta)$ values for weight enumerators $W_{84,1}$ and $W_{84,2}$ can be found online at \cite{wepd} (see \cite{R-058,R-085}). It is unknown whether or not a code with weight enumerator $W_{84,3}$ has been previously reported.

We obtain 107 new best known binary self-dual codes of length 84 which have weight enumerator $W_{84,3}$ for
	\begin{enumerate}[label=]
	\item $\beta=0$ and $\alpha\in\{6z:z=336,\lb350,\lb358,\lb365,\lb372,\lb386,\lb392,\lb393,\lb399,\lb400,\lb406,\lb407,\lb413,\lb414,\lb420,\lb421,\lb427,\lb428,\lb434,\lb435,\lb441,\lb442,\lb448,\lb449,\lb455,\lb456,\lb462,\lb463,\lb469,\lb470,\lb476,\lb477,\lb483,\lb484,\lb490,\lb491,\lb497,\lb498,\lb504,\lb505,\lb511,\lb512,\lb518,\lb519,\lb525,\lb526,\lb532,\lb533,\lb539,\lb540,\lb546,\lb553,\lb554,\lb560,\lb567\}$;
	\item $\beta=21$ and $\alpha\in\{6z:z=413,\lb434,\lb435,\lb441,\lb442,\lb449,\lb455,\lb456,\lb462,\lb463,\lb469,\lb470,\lb476,\lb477,\lb483,\lb484,\lb490,\lb491,\lb497,\lb498,\lb504,\lb505,\lb511,\lb512,\lb518,\lb519,\lb525,\lb526,\lb532,\lb533,\lb539,\lb540,\lb546,\lb547,\lb553,\lb560,\lb568,\lb575,\lb595\}$;
	\item $\beta=42$ and $\alpha\in\{6z:z=490,\lb512,\lb518,\lb525,\lb526,\lb539,\lb540,\lb547,\lb553,\lb560,\lb568\}$;
	\item $\beta=63$ and $\alpha\in\{6z:z=574,575\lb\}$.
	\end{enumerate}

Of the 107 new codes, 55 are constructed by applying Theorem \ref{theorem-Omega_42_1} over $\FF_2$ (Table \ref{table-84-Omega_42_1-F2}) and 52 are constructed by applying Theorem \ref{theorem-Omega_42_2} over $\FF_2$ (Table \ref{table-84-Omega_42_2-F2}). In Tables \ref{table-84-Omega_42_1-F2} and \ref{table-84-Omega_42_2-F2}, we only list 10 codes to save space. We refer to Database 2 of \cite{gmd} for the remaining unlisted codes.

	\begin{table}[h!]
	\caption{New binary self-dual $[84,42,14]$ codes from Theorem \ref{theorem-Omega_42_1} over $\FF_2$ (see Database 2 of \cite{gmd} for codes $\mathcal{C}_{84,11}$ to $\mathcal{C}_{84,55}$).}\label{table-84-Omega_42_1-F2}
	\centering
	\begin{adjustbox}{max width=\textwidth}
	\footnotesize
	\begin{tabular}{cccccc}\midrule
	$\mathcal{C}_{84,i}$ & $\vct{v}$ & $W_{84,j}$ & $\alpha$ & $\beta$ & $|\aut{\mathcal{C}_{84,i}}|$\\\midrule	
	1 & \texttt{(110001110100101111010000011100010000011111)} & $3$ & $2988$ & $0$ & $2\cdot 3\cdot 7$\\
	2 & \texttt{(111111011111011000011010010000101000001001)} & $3$ & $3024$ & $0$ & $2\cdot 3\cdot 7$\\
	3 & \texttt{(001001101100110111101011010000011100011010)} & $3$ & $3030$ & $0$ & $2\cdot 3\cdot 7$\\
	4 & \texttt{(101111111010011001101100101011000001001000)} & $3$ & $3066$ & $0$ & $2\cdot 3\cdot 7$\\
	5 & \texttt{(101001011101011110110100111111001011010100)} & $3$ & $3072$ & $0$ & $2\cdot 3\cdot 7$\\
	6 & \texttt{(111100010111001011001010011100110100001001)} & $3$ & $3108$ & $0$ & $2\cdot 3\cdot 7$\\
	7 & \texttt{(110101110100001100100000110101010010101111)} & $3$ & $3114$ & $0$ & $2\cdot 3\cdot 7$\\
	8 & \texttt{(000000000110000110110010101101100110111000)} & $3$ & $3150$ & $0$ & $2\cdot 3\cdot 7$\\
	9 & \texttt{(101010111001111011101001100100110100100000)} & $3$ & $3156$ & $0$ & $2\cdot 3\cdot 7$\\
	10 & \texttt{(101100110111001110010100000010110101111000)} & $3$ & $3192$ & $0$ & $2\cdot 3\cdot 7$\\\midrule
	\end{tabular}
	\end{adjustbox}
	\end{table}

	\begin{table}[h!]
	\caption{New binary self-dual $[84,42,14]$ codes from Theorem \ref{theorem-Omega_42_2} over $\FF_2$ (see Database 2 of \cite{gmd} for codes $\mathcal{C}_{84,66}$ to $\mathcal{C}_{84,107}$).}\label{table-84-Omega_42_2-F2}
	\centering
	\begin{adjustbox}{max width=\textwidth}
	\footnotesize
	\begin{tabular}{cccccc}\midrule
	$\mathcal{C}_{84,i}$ & $\vct{v}$ & $W_{84,j}$ & $\alpha$ & $\beta$ & $|\aut{\mathcal{C}_{84,i}}|$\\\midrule
	56 & \texttt{(011001100101000010101000000000011110111100)} & $3$ & $2016$ & $0$ & $2^{2}\cdot 3\cdot 7$\\
	57 & \texttt{(100101010001111110100110011011000001011001)} & $3$ & $2100$ & $0$ & $2\cdot 3\cdot 7$\\
	58 & \texttt{(010110001001010100011100001111000111011011)} & $3$ & $2148$ & $0$ & $2\cdot 3\cdot 7$\\
	59 & \texttt{(010101100111000010011001000001000000000001)} & $3$ & $2190$ & $0$ & $2\cdot 3\cdot 7$\\
	60 & \texttt{(101110110000001010001011111001000000000101)} & $3$ & $2232$ & $0$ & $2\cdot 3\cdot 7$\\
	61 & \texttt{(001000100010110011001101111011001001111100)} & $3$ & $2316$ & $0$ & $2\cdot 3\cdot 7$\\
	62 & \texttt{(010010101101010100100111001011011001110001)} & $3$ & $2352$ & $0$ & $2\cdot 3\cdot 7$\\
	63 & \texttt{(001101000100110000001101011011011011110011)} & $3$ & $2358$ & $0$ & $2\cdot 3\cdot 7$\\
	64 & \texttt{(000011000001100101110100001010111101110111)} & $3$ & $2394$ & $0$ & $2\cdot 3\cdot 7$\\
	65 & \texttt{(011101100100110011000111001110111101000000)} & $3$ & $2400$ & $0$ & $2\cdot 3\cdot 7$\\\midrule
	\end{tabular}
	\end{adjustbox}
	\end{table}

\subsection{New Self-Dual Codes of Length 96}

The possible weight enumerators of a singly-even binary self-dual $[96,48,16]$ code are given in \cite{R-059} as
	\begin{align*}
	W_{96,1}^{\text{I}}&=1+(\alpha-5814)x^{16}+(97280+64\beta)x^{18}\\&\quad+(1784320-16\alpha-384\beta)x^{20}\\&\quad+(17626112+192\beta)x^{22}+\cdots,\\
	W_{96,2}^{\text{I}}&=1+(\alpha-5814)x^{16}+(97280+64\beta)x^{18}\\&\quad+(1694208-16\alpha-384\beta+4096\gamma)x^{20}\\&\quad+(18969600+192\beta-49152\gamma)x^{22}+\cdots,
	\end{align*}
where $\alpha,\beta,\gamma\in\ZZ$. Previously known $(\alpha,\beta,\gamma)$ values for weight enumerators $W_{96,1}^{\text{I}}$ and $W_{96,2}^{\text{I}}$ can be found online at \cite{wepd} (see \cite{R-137,R-059}).

We obtain 105 new best known singly-even binary self-dual codes of length 96 which have weight enumerator $W_{96,2}^{\text{I}}$ for
	\begin{enumerate}[label=]
	\item $\gamma=0$ and $(\alpha,\beta)\in\{(12z_1,-4z_2):(z_1,z_2)=(850,\lb0),\lb(896,\lb0),\lb(904,\lb0),\lb(805,\lb1),\lb(854,\lb3),\lb(808,\lb4),\lb(837,\lb6),\lb(926,\lb6),\lb(822,\lb7),\lb(865,\lb9),\lb(860,\lb10),\lb(860,\lb12),\lb(897,\lb12),\lb(900,\lb12),\lb(929,\lb12),\lb(1014,\lb12),\lb(877,\lb13),\lb(910,\lb15),\lb(877,\lb16),\lb(933,\lb18),\lb(908,\lb19),\lb(938,\lb21),\lb(952,\lb22),\lb(957,\lb24),\lb(990,\lb24),\lb(965,\lb25),\lb(1003,\lb27),\lb(947,\lb28),\lb(1038,\lb30),\lb(1052,\lb31),\lb(971,\lb34),\lb(1045,\lb36),\lb(1222,\lb36),\lb(1148,\lb46),\lb(1244,\lb48),\lb(1260,\lb48),\lb(1204,\lb52),\lb(1278,\lb60)\}$;
	\item $\gamma=6$ and $(\alpha,\beta)\in\{(12z_1,-4z_2):(z_1,z_2)=(909,\lb30),\lb(913,\lb31),\lb(922,\lb33),\lb(901,\lb34),\lb(902,\lb36),\lb(918,\lb37),\lb(944,\lb39),\lb(948,\lb40),\lb(932,\lb42),\lb(995,\lb43),\lb(949,\lb45),\lb(980,\lb46),\lb(1034,\lb48),\lb(1018,\lb49),\lb(969,\lb51),\lb(978,\lb52),\lb(1120,\lb64)\}$;
	\item $\gamma=12$ and $(\alpha,\beta)\in\{(12z_1,-4z_2):(z_1,z_2)=(928,\lb60),\lb(988,\lb60),\lb(992,\lb60),\lb(1048,\lb60),\lb(1056,\lb60),\lb(1076,\lb60),\lb(1096,\lb60),\lb(1104,\lb60),\lb(1120,\lb60),\lb(1148,\lb60),\lb(1160,\lb60),\lb(1168,\lb60),\lb(1176,\lb60),\lb(1208,\lb60),\lb(1216,\lb60),\lb(1232,\lb60),\lb(1240,\lb60),\lb(1264,\lb60),\lb(1280,\lb60),\lb(1288,\lb60),\lb(1320,\lb60),\lb(1336,\lb60),\lb(1520,\lb60),\lb(982,\lb61),\lb(975,\lb63),\lb(984,\lb64),\lb(997,\lb66),\lb(1133,\lb66),\lb(1148,\lb66),\lb(1236,\lb66),\lb(977,\lb67),\lb(1075,\lb69),\lb(1042,\lb70),\lb(1080,\lb72),\lb(1112,\lb72),\lb(1120,\lb72),\lb(1137,\lb72),\lb(1272,\lb72),\lb(1544,\lb72),\lb(1036,\lb73),\lb(1046,\lb76),\lb(1098,\lb78),\lb(1121,\lb78),\lb(1072,\lb79),\lb(1226,\lb84),\lb(1352,\lb84),\lb(1528,\lb84),\lb(1224,\lb85),\lb(1332,\lb100),\lb(1384,\lb108)\}$.
	\end{enumerate}

Of the 105 new codes, 5 are constructed by applying Theorem \ref{theorem-Omega_24_1} over $\FF_2+u\FF_2$ (Table \ref{table-96(I)-Omega_24_1-F2u}); 56 are constructed by applying Theorem \ref{theorem-Omega_24_1} over $\FF_4$ (Table \ref{table-96(I)-Omega_24_1-F4}); 29 are constructed by applying Theorem \ref{theorem-Omega_24_2} over $\FF_2+u\FF_2$ (Table \ref{table-96(I)-Omega_24_2-F2u}) and 15 are constructed by applying Theorem \ref{theorem-Omega_24_3} over $\FF_2+u\FF_2$ (Table \ref{table-96(I)-Omega_24_3-F2u}). In Tables \ref{table-96(I)-Omega_24_1-F4} and \ref{table-96(I)-Omega_24_2-F2u}, we only list 10 codes to save space. We refer to Database 3 of \cite{gmd} for the remaining unlisted codes.

	\begin{table}[h!]
	\caption{New singly-even binary self-dual $[96,48,16]$ codes from Theorem \ref{theorem-Omega_24_1} over $\FF_2+u\FF_2$.}\label{table-96(I)-Omega_24_1-F2u}
	\centering
	\begin{adjustbox}{max width=\textwidth}
	\footnotesize
	\begin{tabular}{cccccc}\midrule
	$\mathcal{C}_{96,i}^{\text{I}}$ & $\vct{v}$ & $\alpha$ & $\beta$ & $\gamma$ & $|\aut{\mathcal{C}_{96,i}^{\text{I}}}|$\\\midrule
	1 & \texttt{(021111013112231302031321)} & $15336$ & $-240$ & $0$ & $2^{4}\cdot 3$\\
	2 & \texttt{(332030221021223333303031)} & $14664$ & $-144$ & $0$ & $2^{4}\cdot 3$\\
	3 & \texttt{(310201300213103023131203)} & $12456$ & $-120$ & $0$ & $2^{4}\cdot 3$\\
	4 & \texttt{(110330330331133112022003)} & $16608$ & $-432$ & $12$ & $2^{6}\cdot 3$\\
	5 & \texttt{(301201202300231031203031)} & $14712$ & $-336$ & $12$ & $2^{5}\cdot 3$\\\midrule
	\end{tabular}
	\end{adjustbox}
	\end{table}
	
	\begin{table}[h!]
	\caption{New singly-even binary self-dual $[96,48,16]$ codes from Theorem \ref{theorem-Omega_24_1} over $\FF_4$ (see Database 3 of \cite{gmd} for codes $\mathcal{C}_{96,16}^{\text{I}}$ to $\mathcal{C}_{96,61}^{\text{I}}$).}\label{table-96(I)-Omega_24_1-F4}
	\centering
	\begin{adjustbox}{max width=\textwidth}
	\footnotesize
	\begin{tabular}{cccccc}\midrule
	$\mathcal{C}_{96,i}^{\text{I}}$ & $\vct{v}$ & $\alpha$ & $\beta$ & $\gamma$ &  $|\aut{\mathcal{C}_{96,i}^{\text{I}}}|$\\\midrule
	6 & \texttt{(301220102333222223210331)} & $14448$ & $-208$ & $0$ & $2^{4}\cdot 3$\\
	7 & \texttt{(111322103200321233201211)} & $13776$ & $-184$ & $0$ & $2^{4}\cdot 3$\\
	8 & \texttt{(333110012302102113330110)} & $11652$ & $-136$ & $0$ & $2^{4}\cdot 3$\\
	9 & \texttt{(321212110001220122211301)} & $12624$ & $-124$ & $0$ & $2^{3}\cdot 3$\\
	10 & \texttt{(000232332103103311032121)} & $11364$ & $-112$ & $0$ & $2^{3}\cdot 3$\\
	11 & \texttt{(231232002131031220200120)} & $12036$ & $-108$ & $0$ & $2^{3}\cdot 3$\\
	12 & \texttt{(021301113010112220211130)} & $11580$ & $-100$ & $0$ & $2^{3}\cdot 3$\\
	13 & \texttt{(000332130323021220110022)} & $11880$ & $-96$ & $0$ & $2^{3}\cdot 3$\\
	14 & \texttt{(001022122300133130333310)} & $11424$ & $-88$ & $0$ & $2^{3}\cdot 3$\\
	15 & \texttt{(213121322231133130230323)} & $11256$ & $-84$ & $0$ & $2^{3}\cdot 3$\\\midrule
	\end{tabular}
	\end{adjustbox}
	\end{table}
	
	\begin{table}[h!]
	\caption{New singly-even binary self-dual $[96,48,16]$ codes from Theorem \ref{theorem-Omega_24_2} over $\FF_2+u\FF_2$ (see Database 3 of \cite{gmd} for codes $\mathcal{C}_{96,72}^{\text{I}}$ to $\mathcal{C}_{96,90}^{\text{I}}$).}\label{table-96(I)-Omega_24_2-F2u}
	\centering
	\begin{adjustbox}{max width=\textwidth}
	\footnotesize
	\begin{tabular}{ccccccc}\midrule
	$\mathcal{C}_{96,i}^{\text{I}}$ & $\vct{v}$ & $\alpha$ & $\beta$ & $\gamma$ &  $|\aut{\mathcal{C}_{96,i}^{\text{I}}}|$\\\midrule
	62 & \texttt{(222222222220220133213123)} & $14928$ & $-192$ & $0$ & $2^{6}\cdot 3$\\
	63 & \texttt{(222222222220220133211121)} & $15120$ & $-192$ & $0$ & $2^{6}\cdot 3$\\
	64 & \texttt{(222220222011020210021113)} & $12540$ & $-144$ & $0$ & $2^{4}\cdot 3$\\
	65 & \texttt{(222222222011021013011303)} & $11484$ & $-96$ & $0$ & $2^{4}\cdot 3$\\
	66 & \texttt{(222222222011202110211131)} & $10764$ & $-48$ & $0$ & $2^{4}\cdot 3$\\
	67 & \texttt{(222222222101200131212230)} & $10800$ & $-48$ & $0$ & $2^{5}\cdot 3$\\
	68 & \texttt{(222222222011202110213111)} & $11148$ & $-48$ & $0$ & $2^{4}\cdot 3$\\
	69 & \texttt{(222222220103021223012121)} & $12168$ & $-48$ & $0$ & $2^{4}\cdot 3$\\
	70 & \texttt{(222222222101200131221023)} & $10752$ & $0$ & $0$ & $2^{5}\cdot 3$\\
	71 & \texttt{(222222202121200111221203)} & $10848$ & $0$ & $0$ & $2^{5}\cdot 3$\\\midrule
	\end{tabular}
	\end{adjustbox}
	\end{table}
	
	\begin{table}[h!]
	\caption{New singly-even binary self-dual $[96,48,16]$ codes from Theorem \ref{theorem-Omega_24_3} over $\FF_2+u\FF_2$.}\label{table-96(I)-Omega_24_3-F2u}
	\centering
	\begin{adjustbox}{max width=\textwidth}
	\footnotesize
	\begin{tabular}{ccccccc}\midrule
	$\mathcal{C}_{96,i}^{\text{I}}$ & $\vct{v}$ & $\alpha$ & $\beta$ & $\gamma$ & $|\aut{\mathcal{C}_{96,i}^{\text{I}}}|$\\\midrule
	91 & \texttt{(222220222111201001210311)} & $11112$ & $-24$ & $0$ & $2^{4}\cdot 3$\\
	92 & \texttt{(222222202111001223010313)} & $16224$ & $-336$ & $12$ & $2^{6}\cdot 3$\\
	93 & \texttt{(222222222011101333122333)} & $18336$ & $-336$ & $12$ & $2^{5}\cdot 3$\\
	94 & \texttt{(222222222101220113212010)} & $15264$ & $-288$ & $12$ & $2^{5}\cdot 3$\\
	95 & \texttt{(222220222111221003212111)} & $18528$ & $-288$ & $12$ & $2^{6}\cdot 3$\\
	96 & \texttt{(222220220101211331210113)} & $14832$ & $-264$ & $12$ & $2^{4}\cdot 3$\\
	97 & \texttt{(222220200103211331212113)} & $13776$ & $-240$ & $12$ & $2^{4}\cdot 3$\\
	98 & \texttt{(222220222111221003212313)} & $13920$ & $-240$ & $12$ & $2^{5}\cdot 3$\\
	99 & \texttt{(222222220021202121211101)} & $14496$ & $-240$ & $12$ & $2^{6}\cdot 3$\\
	100 & \texttt{(222222222113212131201003)} & $14592$ & $-240$ & $12$ & $2^{5}\cdot 3$\\
	101 & \texttt{(222222222011212313201101)} & $14784$ & $-240$ & $12$ & $2^{5}\cdot 3$\\
	102 & \texttt{(222222220021020101011121)} & $14880$ & $-240$ & $12$ & $2^{5}\cdot 3$\\
	103 & \texttt{(222222222113212333201003)} & $15360$ & $-240$ & $12$ & $2^{5}\cdot 3$\\
	104 & \texttt{(222222222011011111002013)} & $15456$ & $-240$ & $12$ & $2^{5}\cdot 3$\\
	105 & \texttt{(222222222211020101021321)} & $16032$ & $-240$ & $12$ & $2^{5}\cdot 3$\\\midrule
	\end{tabular}
	\end{adjustbox}
	\end{table}

The weight enumerator of a doubly-even binary self-dual $[96,48,16]$ code is given in \cite{R-059} as
	\begin{align*}
	W_{96}^{\text{II}}&=1+\alpha x^{16}+(3217056-16\alpha)x^{20}+\cdots,
	\end{align*}
where $\alpha\in\ZZ$. Previously known $\alpha$ values for weight enumerator $W_{96}^{\text{II}}$ can be found online at \cite{wepd} (see \cite{R-177,R-044,R-170,R-174,R-048,R-172,R-126,R-059}).

We obtain 121 new best known doubly-even binary self-dual codes of length 96 which have weight enumerator $W_{96}^{\text{II}}$ for
	\begin{enumerate}[label=]
	\item $\alpha\in\{6z:z=1379,\lb1403,\lb1419,\lb1443,\lb1459,\lb1473,\lb1499,\lb1507,\lb1523,\lb1539,\lb1547,\lb1563,\lb1579,\lb1603,\lb1619,\lb1627,\lb1643,\lb1659,\lb1667,\lb1683,\lb1699,\lb1707,\lb1723,\lb1747,\lb1759,\lb1763,\lb1779,\lb1787,\lb1795,\lb1803,\lb1811,\lb1819,\lb1827,\lb1835,\lb1843,\lb1851,\lb1859,\lb1867,\lb1875,\lb1879,\lb1883,\lb1891,\lb1899,\lb1903,\lb1907,\lb1913,\lb1915,\lb1921,\lb1923,\lb1931,\lb1939,\lb1947,\lb1957,\lb1963,\lb1971,\lb1975,\lb1979,\lb1987,\lb1995,\lb2003,\lb2007,\lb2011,\lb2015,\lb2019,\lb2023,\lb2027,\lb2031,\lb2039,\lb2043,\lb2055,\lb2059,\lb2067,\lb2071,\lb2079,\lb2083,\lb2087,\lb2091,\lb2095,\lb2103,\lb2107,\lb2119,\lb2127,\lb2135,\lb2143,\lb2147,\lb2151,\lb2163,\lb2167,\lb2175,\lb2195,\lb2199,\lb2203,\lb2207,\lb2211,\lb2215,\lb2223,\lb2231,\lb2247,\lb2255,\lb2259,\lb2263,\lb2279,\lb2283,\lb2295,\lb2311,\lb2359,\lb2379,\lb2407,\lb2423,\lb2471,\lb2483,\lb2503,\lb2519,\lb2567,\lb2599,\lb2663,\lb2695,\lb2711,\lb2759,\lb2887,\lb4751\}$.
	\end{enumerate}

Of the 121 new codes, 88 are constructed by applying Theorem \ref{theorem-Omega_24_1} over $\FF_2+u\FF_2$ (Table \ref{table-96(II)-Omega_24_1-F2u}); 8 are constructed by applying Theorem \ref{theorem-Omega_24_1} over $\FF_4$ (Table \ref{table-96(II)-Omega_24_1-F4}); 13 are constructed by applying Theorem \ref{theorem-Omega_24_2} over $\FF_2+u\FF_2$ (Table \ref{table-96(II)-Omega_24_2-F2u}) and 12 are constructed by applying Theorem \ref{theorem-Omega_24_3} over $\FF_2+u\FF_2$ (Table \ref{table-96(II)-Omega_24_3-F2u}). In Table \ref{table-96(II)-Omega_24_1-F2u}, we only list 10 codes to save space. We refer to Database 4 of \cite{gmd} for the remaining unlisted codes.

	\begin{table}[h!]
	\caption{New doubly-even binary self-dual $[96,48,16]$ codes from Theorem \ref{theorem-Omega_24_1} over $\FF_2+u\FF_2$ (see Database 4 of \cite{gmd} for codes $\mathcal{C}_{96,11}^{\text{II}}$ to $\mathcal{C}_{96,88}^{\text{II}}$).}\label{table-96(II)-Omega_24_1-F2u}
	\centering
	\begin{adjustbox}{max width=\textwidth}
	\footnotesize
	\begin{tabular}{cccc}\midrule
	$\mathcal{C}_{96,i}^{\text{II}}$ & $\vct{v}$ & $\alpha$ & $|\aut{\mathcal{C}_{96,i}^{\text{II}}}|$\\\midrule
	1 & \texttt{(320210300223213323022021)} & $8514$ & $2^{4}\cdot 3$\\
	2 & \texttt{(122313111112022110302021)} & $8754$ & $2^{4}\cdot 3$\\
	3 & \texttt{(122123010133300221011031)} & $8994$ & $2^{4}\cdot 3$\\
	4 & \texttt{(001212011312020203212003)} & $9042$ & $2^{4}\cdot 3$\\
	5 & \texttt{(122000032021320000301313)} & $9138$ & $2^{4}\cdot 3$\\
	6 & \texttt{(010220032021103212312322)} & $9234$ & $2^{4}\cdot 3$\\
	7 & \texttt{(210231130330223123221020)} & $9282$ & $2^{4}\cdot 3$\\
	8 & \texttt{(032311303332300120032321)} & $9378$ & $2^{4}\cdot 3$\\
	9 & \texttt{(213201111011203112303130)} & $9474$ & $2^{4}\cdot 3$\\
	10 & \texttt{(110230310113303323101232)} & $9618$ & $2^{4}\cdot 3$\\\midrule
	\end{tabular}
	\end{adjustbox}
	\end{table}

	\begin{table}[h!]
	\caption{New doubly-even binary self-dual $[96,48,16]$ codes from Theorem \ref{theorem-Omega_24_1} over $\FF_4$.}\label{table-96(II)-Omega_24_1-F4}
	\centering
	\begin{adjustbox}{max width=\textwidth}
	\footnotesize
	\begin{tabular}{cccc}\midrule
	$\mathcal{C}_{96,i}^{\text{II}}$ & $\vct{v}$ & $\alpha$ & $|\aut{\mathcal{C}_{96,i}^{\text{II}}}|$\\\midrule
	89 & \texttt{(332010230212013330233103)} & $8274$ & $2^{3}\cdot 3$\\
	90 & \texttt{(121001211131002223313030)} & $8418$ & $2^{3}\cdot 3$\\
	91 & \texttt{(330222312102031223221213)} & $8658$ & $2^{3}\cdot 3$\\
	92 & \texttt{(331001322120111003113202)} & $8838$ & $2^{3}\cdot 3$\\
	93 & \texttt{(322112032202123203331221)} & $11478$ & $2^{3}\cdot 3$\\
	94 & \texttt{(333003302201123232100313)} & $11526$ & $2^{3}\cdot 3$\\
	95 & \texttt{(201120113100000113122122)} & $11742$ & $2^{3}\cdot 3$\\
	96 & \texttt{(000232210010130121123202)} & $13194$ & $2^{5}\cdot 3$\\\midrule
	\end{tabular}
	\end{adjustbox}
	\end{table}
	
	\begin{table}[h!]
	\caption{New doubly-even binary self-dual $[96,48,16]$ codes from Theorem \ref{theorem-Omega_24_2} over $\FF_2+u\FF_2$.}\label{table-96(II)-Omega_24_2-F2u}
	\centering
	\begin{adjustbox}{max width=\textwidth}
	\footnotesize
	\begin{tabular}{cccc}\midrule
	$\mathcal{C}_{96,i}^{\text{II}}$ & $\vct{v}$ & $\alpha$ & $|\aut{\mathcal{C}_{96,i}^{\text{II}}}|$\\\midrule
	97 & \texttt{(222222220103200133210030)} & $10002$ & $2^{4}\cdot 3$\\
	98 & \texttt{(222222220103021003012303)} & $10098$ & $2^{4}\cdot 3$\\
	99 & \texttt{(222222220103200133212032)} & $10578$ & $2^{4}\cdot 3$\\
	100 & \texttt{(222222220103021003010123)} & $10818$ & $2^{4}\cdot 3$\\
	101 & \texttt{(222222220103221203210101)} & $10866$ & $2^{4}\cdot 3$\\
	102 & \texttt{(222220202013221331211111)} & $12138$ & $2^{5}\cdot 3$\\
	103 & \texttt{(222222222101122211113131)} & $12234$ & $2^{6}\cdot 3$\\
	104 & \texttt{(222222222101020131001221)} & $12522$ & $2^{6}\cdot 3$\\
	105 & \texttt{(222222220103200113212230)} & $12546$ & $2^{4}\cdot 3$\\
	106 & \texttt{(222222222101201021210101)} & $12810$ & $2^{6}\cdot 3$\\
	107 & \texttt{(222220222211020212001111)} & $13290$ & $2^{6}\cdot 3$\\
	108 & \texttt{(222222202013220110213131)} & $13578$ & $2^{5}\cdot 3$\\
	109 & \texttt{(222222222220222111213123)} & $28506$ & $2^{8}\cdot 3\cdot 5$\\\midrule
	\end{tabular}
	\end{adjustbox}
	\end{table}
	
	\begin{table}[h!]
	\caption{New doubly-even binary self-dual $[96,48,16]$ codes from Theorem \ref{theorem-Omega_24_3} over $\FF_2+u\FF_2$.}\label{table-96(II)-Omega_24_3-F2u}
	\centering
	\begin{adjustbox}{max width=\textwidth}
	\footnotesize
	\begin{tabular}{cccc}\midrule
	$\mathcal{C}_{96,i}^{\text{II}}$ & $\vct{v}$ & $\alpha$ & $|\aut{\mathcal{C}_{96,i}^{\text{II}}}|$\\\midrule
	110 & \texttt{(222220200103011331010113)} & $12186$ & $2^{4}\cdot 3$\\
	111 & \texttt{(222222222011202121201123)} & $12426$ & $2^{5}\cdot 3$\\
	112 & \texttt{(222222222011211111220213)} & $12714$ & $2^{5}\cdot 3$\\
	113 & \texttt{(222220220101011331012113)} & $12762$ & $2^{4}\cdot 3$\\
	114 & \texttt{(222222222011212313221303)} & $13002$ & $2^{5}\cdot 3$\\
	115 & \texttt{(222020200101011331012133)} & $13050$ & $2^{4}\cdot 3$\\
	116 & \texttt{(222220220101011331012133)} & $13338$ & $2^{4}\cdot 3$\\
	117 & \texttt{(222222220211121333100113)} & $13866$ & $2^{6}\cdot 3$\\
	118 & \texttt{(222222220211121333100131)} & $14826$ & $2^{6}\cdot 3$\\
	119 & \texttt{(222222222011211311220213)} & $15978$ & $2^{6}\cdot 3$\\
	120 & \texttt{(222222222011121333100333)} & $16170$ & $2^{5}\cdot 3$\\
	121 & \texttt{(222222222011121333100311)} & $16554$ & $2^{5}\cdot 3$\\\midrule
	\end{tabular}
	\end{adjustbox}
	\end{table}

\section{Conclusion}

In this work, we applied the idea of composite matrices $\Omega(v)$ to derive a number of techniques assuming a generator matrix of the form $(I_n\,|\,\Omega(v))$ to construct new binary self-dual codes. We defined each of the composite matrices that were implemented in the techniques and we proved the necessary conditions required by the techniques to produce self-dual codes. We applied these techniques directly over $\FF_2$ as well as over the rings $\FF_2+u\FF_2$ and $\FF_4$. By so doing, we were able to construct new best known binary self-dual codes with many different weight enumerator parameter values. In particular, we constructed 28 singly-even $[80,40,14]$ codes, 107 $[84,42,14]$ codes, 105 singly-even $[96,48,16]$ codes and 121 doubly-even $[96,48,16]$ codes.

The advantage of using composite matrices is that there are many different combinations of their determining parameters, i.e. the groups $G$ and $\{H_1,H_2,\ldots,H_{\eta}\}$ and the parameter matrices $H'$ and $P'$. This allows for many different forms of the matrices $\Omega(v)$ which often have very unusual structures. For each of the composite matrices we defined, we assumed that $H'=\vct{1}$ and $P'=\vct{1}$. A suggestion for future work could be to investigate different choices for both $H'$ and $P'$. Another suggestion would be to use composite matrices determined by groups $G$ and $\{H_1,H_2,\ldots,H_{\eta}\}$ of different orders. We could also investigate applying composite matrices over rings other than those used in this work.

\bibliographystyle{plainnat}
\bibliography{paper3}

\begin{thebibliography}{32}
\providecommand{\natexlab}[1]{#1}
\providecommand{\url}[1]{\texttt{#1}}
\expandafter\ifx\csname urlstyle\endcsname\relax
  \providecommand{\doi}[1]{doi: #1}\else
  \providecommand{\doi}{doi: \begingroup \urlstyle{rm}\Url}\fi

\bibitem[Bortos et~al.(2020)Bortos, Gildea, Kaya, Korban, and
  Tylyshchak]{R-100}
M.~Bortos, J.~Gildea, A.~Kaya, A.~Korban, and A.~Tylyshchak.
\newblock {New self-dual codes of length 68 from a $2\times 2$ block matrix
  construction and group rings}.
\newblock \emph{{Adv. Math. Commun.}}, 2020.
\newblock \doi{10.3934/amc.2020111}.

\bibitem[Bosma et~al.(1997)Bosma, Cannon, and Playoust]{magma}
W.~Bosma, J.~Cannon, and C.~Playoust.
\newblock {The Magma Algebra System I: The User Language}.
\newblock \emph{{J. Symbolic Comput.}}, 24\penalty0 (3--4):\penalty0 235--265,
  1997.
\newblock \doi{10.1006/jsco.1996.0125}.

\bibitem[Bouyukliev(2007)]{Q-extension}
I.~G. Bouyukliev.
\newblock {What is Q-extension?}
\newblock \emph{{Serdica J. Comput.}}, 1\penalty0 (2):\penalty0 115--130, 2007.

\bibitem[Dontcheva(2002)]{R-170}
R.~Dontcheva.
\newblock {On the Doubly-Even Self-Dual Codes of Length 96}.
\newblock \emph{{IEEE Trans. Inform. Theory}}, 48\penalty0 (2):\penalty0
  557--561, 2002.
\newblock \doi{10.1109/18.979333}.

\bibitem[Dougherty et~al.(1997)Dougherty, Gulliver, and Harada]{R-044}
S.~T. Dougherty, T.~A. Gulliver, and M.~Harada.
\newblock {Extremal Binary Self-Dual Codes}.
\newblock \emph{{IEEE Trans. Inform. Theory}}, 43\penalty0 (6):\penalty0
  2036--2047, 1997.
\newblock \doi{10.1109/18.641574}.

\bibitem[Dougherty et~al.(1999)Dougherty, Gaborit, Harada, and Sol\'{e}]{R-117}
S.~T. Dougherty, P.~Gaborit, M.~Harada, and P.~Sol\'{e}.
\newblock {Type II Codes over $\FF_2+u\FF_2$}.
\newblock \emph{{IEEE Trans. Inform. Theory}}, 45\penalty0 (1):\penalty0
  32--45, 1999.
\newblock \doi{10.1109/18.746770}.

\bibitem[Dougherty et~al.(2020{\natexlab{a}})Dougherty, Gildea, and
  Kaya]{R-111}
S.~T. Dougherty, J.~Gildea, and A.~Kaya.
\newblock {$2^n$ Bordered constructions of self-dual codes from group rings}.
\newblock \emph{{Finite Fields Appl.}}, 67, 2020{\natexlab{a}}.
\newblock \doi{doi.org/10.1016/j.ffa.2020.101692}.

\bibitem[Dougherty et~al.(2020{\natexlab{b}})Dougherty, Gildea, Korban, and
  Kaya]{R-101}
S.~T. Dougherty, J.~Gildea, A.~Korban, and A.~Kaya.
\newblock {Composite Matrices from Group Rings, Composite G-Codes and
  Constructions of Self-Dual Codes}, 2020{\natexlab{b}}.
\newblock \url{https://arxiv.org/abs/2002.11614}.

\bibitem[Dougherty et~al.(2020{\natexlab{c}})Dougherty, Gildea, Korban, and
  Kaya]{R-151}
S.~T. Dougherty, J.~Gildea, A.~Korban, and A.~Kaya.
\newblock {New extremal self-dual binary codes of length 68 via composite
  construction, $\FF_2+u\FF_2$ lifts, extensions and neighbours}.
\newblock \emph{{Int. J. Inf. Coding Theory}}, 5\penalty0 (3--4):\penalty0
  211--226, 2020{\natexlab{c}}.
\newblock \doi{10.1504/IJICOT.2020.110703}.

\bibitem[Dougherty et~al.(2021)Dougherty, Gildea, and Korban]{R-153}
S.~T. Dougherty, J.~Gildea, and A.~Korban.
\newblock {Extending an established isomorphism between group rings and a
  subring of the $n\times n$ matrices}.
\newblock \emph{{Internat. J. Algebra Comput.}}, 2021.
\newblock \doi{10.1142/S0218196721500223}.

\bibitem[Feit(1974)]{R-177}
W.~Feit.
\newblock {A Self-Dual Even $(96,48,16)$ Code}.
\newblock \emph{{IEEE Trans. Inform. Theory}}, 20\penalty0 (1):\penalty0
  136--138, 1974.
\newblock \doi{10.1109/TIT.1974.1055153}.

\bibitem[Gaborit et~al.(2002)Gaborit, Pless, Sol\'{e}, and Atkin]{R-118}
P.~Gaborit, V.~Pless, P.~Sol\'{e}, and O.~Atkin.
\newblock {Type II Codes over $\FF_4$}.
\newblock \emph{{Finite Fields Appl.}}, 8\penalty0 (2):\penalty0 171--183,
  2002.
\newblock \doi{10.1006/ffta.2001.0333}.

\bibitem[Gildea et~al.(2019)Gildea, Kaya, Tylyshchak, and Yildiz]{R-080}
J.~Gildea, A.~Kaya, A.~Tylyshchak, and B.~Yildiz.
\newblock {A Group Induced Four-circulant Construction for Self-dual Codes and
  New Extremal Binary Self-dual Codes}, 2019.
\newblock \url{https://arxiv.org/abs/1912.11758}.

\bibitem[Gildea et~al.(2020{\natexlab{a}})Gildea, Kaya, Korban, and
  A.Tylyshchak]{R-113}
J.~Gildea, A.~Kaya, A.~Korban, and A.Tylyshchak.
\newblock {Self-dual codes using bisymmetric matrices and group rings}.
\newblock \emph{{Discrete Math.}}, 343\penalty0 (11), 2020{\natexlab{a}}.
\newblock \doi{10.1016/j.disc.2020.112085}.

\bibitem[Gildea et~al.(2020{\natexlab{b}})Gildea, Korban, Kaya, and
  Yildiz]{R-096}
J.~Gildea, A.~Korban, A.~Kaya, and B.~Yildiz.
\newblock {Constructing self-dual codes from group rings and reverse circulant
  matrices}.
\newblock \emph{{Adv. Math. Commun.}}, 2020{\natexlab{b}}.
\newblock \doi{10.3934/amc.2020077}.

\bibitem[Gildea et~al.(2020{\natexlab{c}})Gildea, Taylor, Kaya, and
  Tylyshchak]{R-067}
J.~Gildea, R.~Taylor, A.~Kaya, and A.~Tylyshchak.
\newblock {Double bordered constructions of self-dual codes from group rings
  over Frobenius rings}.
\newblock \emph{{Cryptogr. Commun.}}, 12\penalty0 (4):\penalty0 769--784,
  2020{\natexlab{c}}.
\newblock \doi{10.1007/s12095-019-00420-3}.

\bibitem[Gildea et~al.(2021{\natexlab{a}})Gildea, Korban, and Roberts]{amr1}
J.~Gildea, A.~Korban, and A.~M. Roberts.
\newblock {New binary self-dual codes of lengths 56, 58, 64, 80 and 92 from a
  modification of the four circulant construction}, 2021{\natexlab{a}}.
\newblock \url{https://arxiv.org/abs/2102.10354}.

\bibitem[Gildea et~al.(2021{\natexlab{b}})Gildea, Korban, and Roberts]{gmd}
J.~Gildea, A.~Korban, and A.~M. Roberts.
\newblock Generator matrix database, 2021{\natexlab{b}}.
\newblock \url{https://amr3-ys3da62trb.netlify.app}.

\bibitem[Gulliver and Harada(2006)]{R-058}
T.~A. Gulliver and M.~Harada.
\newblock {Classification of extremal double circulant self-dual codes of
  lengths 74--88}.
\newblock \emph{{Discrete Math.}}, 306\penalty0 (17):\penalty0 2064--2072,
  2006.
\newblock \doi{10.1016/j.disc.2006.05.004}.

\bibitem[Gulliver and Harada(2019)]{R-059}
T.~A. Gulliver and M.~Harada.
\newblock {On extremal double circulant self-dual codes of lengths 90--96}.
\newblock \emph{{Appl. Algebra Engrg. Comm. Comput.}}, 30\penalty0
  (5):\penalty0 403--415, 2019.
\newblock \doi{10.1007/s00200-019-00381-3}.

\bibitem[Harada and Yorgova(2008)]{R-174}
M.~Harada and R.~Yorgova.
\newblock {Construction of a self-dual $[94,47,16]$ code}.
\newblock \emph{{``Eleventh International Workshop on Algebraic and
  Combinatorial Coding Theory'', Pamporovo, Bulgaria}}, pages 125--128, 2008.

\bibitem[Hurley(2006)]{R-022}
T.~Hurley.
\newblock {Group rings and rings of matrices}.
\newblock \emph{{Int. J. Pure Appl. Math.}}, 31\penalty0 (3):\penalty0
  319--335, 2006.

\bibitem[Kaya and Yildiz(2016)]{R-126}
A.~Kaya and B.~Yildiz.
\newblock {Various constructions for self-dual codes over rings and new binary
  self-dual codes}.
\newblock \emph{{Discrete Math.}}, 339\penalty0 (2):\penalty0 460--469, 2016.
\newblock \doi{10.1016/j.disc.2015.09.010}.

\bibitem[Kaya et~al.(2014)Kaya, Yildiz, and Siap]{R-048}
A.~Kaya, B.~Yildiz, and I.~Siap.
\newblock {New extremal binary self-dual codes of length 68 from quadratic
  residue codes over $\FF_2+u\FF_2+u^2\FF_2$}.
\newblock \emph{{Finite Fields Appl.}}, 29:\penalty0 160--177, 2014.
\newblock \doi{10.1016/j.ffa.2014.04.009}.

\bibitem[Korban et~al.(2021)Korban, {\c{S}}ahinkaya, and Ustun]{R-158}
A.~Korban, S.~{\c{S}}ahinkaya, and D.~Ustun.
\newblock {New Type I Binary $[72,36,12]$ Self-Dual Codes from Composite
  Matrices and $R_1$ Lifts}, 2021.
\newblock \url{https://arxiv.org/abs/2102.00474}.

\bibitem[Mallows and Sloane(1973)]{R-116}
C.~L. Mallows and N.~J.~A. Sloane.
\newblock {An Upper Bound for Self-Dual Codes}.
\newblock \emph{{Information and Control}}, 22\penalty0 (2):\penalty0 188--200,
  1973.
\newblock \doi{10.1016/S0019-9958(73)90273-8}.

\bibitem[Rains(1998)]{R-115}
E.~M. Rains.
\newblock {Shadow Bounds for Self-Dual Codes}.
\newblock \emph{{IEEE Trans. Inform. Theory}}, 44\penalty0 (1):\penalty0
  134--139, 1998.
\newblock \doi{10.1109/18.651000}.

\bibitem[Roberts(2020)]{mres}
A.~M. Roberts.
\newblock {Constructions of extremal and optimal self-dual and Hermitian
  self-dual codes over finite fields using circulant matrices}.
\newblock Master's thesis, {University of Chester}, Chester, UK, 2020.
\newblock
  \url{https://drive.google.com/file/d/1CMjnuBvQtrXOY8foy6_gfXOcFFuHAaFs/view}.

\bibitem[Roberts(2021)]{wepd}
A.~M. Roberts.
\newblock {Weight enumerator parameter database for binary self-dual codes},
  2021.
\newblock \url{https://amr-wepd-bsdc.netlify.app}.

\bibitem[Yankov(2014)]{R-172}
N.~Yankov.
\newblock {Some New Self-Dual $[96,48,16]$ Codes with an Automorphism of Order
  15}.
\newblock \emph{{Annual of Konstantin Preslavsky University of Shumen}}, XVI
  C:\penalty0 99--108, 2014.

\bibitem[Yankov et~al.(2017)Yankov, Anev, and G{\"{u}}rel]{R-085}
N.~Yankov, D.~Anev, and M.~G{\"{u}}rel.
\newblock {Self-dual codes with an automorphism of order 13}.
\newblock \emph{{Adv. Math. Commun.}}, 11\penalty0 (3):\penalty0 635--645,
  2017.
\newblock \doi{10.3934/amc.2017047}.

\bibitem[Yorgova and Wassermann(2008)]{R-137}
R.~Yorgova and A.~Wassermann.
\newblock {Binary self-dual codes with automorphisms of order 23}.
\newblock \emph{{Des. Codes Cryptogr.}}, 48\penalty0 (2):\penalty0 155--164,
  2008.
\newblock \doi{10.1007/s10623-007-9152-8}.

\end{thebibliography}
\end{document}